\documentclass[a4paper]{amsart}
\usepackage{amsmath,amsthm,amssymb,latexsym,epic,bbm,comment,yfonts,mathrsfs}
\usepackage{graphicx,enumerate,stmaryrd,rotating}
\usepackage[all]{xy}
\xyoption{2cell}

\allowdisplaybreaks

\usepackage[active]{srcltx}
\usepackage[parfill]{parskip}

\newtheorem{theorem}{Theorem}
\newtheorem{lemma}[theorem]{Lemma}
\newtheorem{remark}[theorem]{Remark}

\newtheorem{corollary}[theorem]{Corollary}
\newtheorem{proposition}[theorem]{Proposition}
\newtheorem{example}[theorem]{Example}

\newtheorem{definition}[theorem]{Definition}

\numberwithin{equation}{section}

\newcommand{\tto}{\twoheadrightarrow}

\newcommand{\Hom}{{\mathrm{Hom}}}
\newcommand{\simp}{{\mathrm{sim}}}
\newcommand{\Nat}{\mathrm{Nat}}

\newcommand{\Mat}{{\mathrm{Mat}}}

\newcommand{\add}{{\mathrm{add}}}

\newcommand{\eins}{\leavevmode\hbox{\small1\kern-3.8pt\normalsize1}}
\newcommand{\op}{{\mathrm{op}}}

\newcommand{\res}{{\rm Res} }

\newcommand{\ind}{{\rm Ind} }
\newcommand{\coind}{{\rm Coind} }

\newcommand{\Ann}{{\rm Ann} }

\newcommand{\mR}{\mathbb{R}}
\newcommand{\mC}{\mathbb{C}}
\newcommand{\mN}{\mathbb{N}}

\newcommand{\mZ}{\mathbb{Z}}

\newcommand{\mm}{\mathsf{m}}

\newcommand{\fa}{{\mathfrak a}}
\newcommand{\fb}{{\mathfrak b}}

\newcommand{\fg}{{\mathfrak g}}
\newcommand{\fh}{{\mathfrak h}}

\newcommand{\fk}{{\mathfrak k}}
\newcommand{\fl}{{\mathfrak l}}
\newcommand{\fn}{{\mathfrak n}}

\newcommand{\fp}{{\mathfrak p}}

\newcommand{\cF}{\mathcal{F}}

\newcommand{\cC}{\mathcal{C}}

\newcommand{\cO}{\mathcal{O}}

\newcommand{\cL}{\mathcal{L}}
\newcommand{\cB}{\mathcal{B}}

\newcommand{\oa}{\bar{0}}
\newcommand{\ob}{\bar{1}}

\newcommand{\End}{{\rm End}}

\newcommand{\pr}{{\rm Pr}}

\newcommand{\fs}{\mathfrak{s}}

\newcommand{\rad}{{\rm rad}}
\newcommand{\Id}{{\rm Id}}

\newcommand{\im}{{\rm im}}

\newcommand{\mk}{\Bbbk}

\newcommand{\wfg}{\widetilde{\mathfrak{g}}}
\newcommand{\gr}{{\rm gr}}

\newcommand{\GK}{\mathsf{GK}}
\newcommand{\BN}{\mathsf{e}}

\newcommand{\supp}{\mathrm{supp}}

  \newcommand{\soc}{\mathrm{soc}}
  \newcommand{\topp}{\mathrm{top}}

  \newcommand{\la}{{\lambda}}
  
\begin{document}

\title[Translated simple modules]{Translated simple modules for Lie algebras
and simple supermodules for Lie superalgebras}

\author{Chih-Whi Chen, Kevin Coulembier and Volodymyr Mazorchuk}
\date{}

\begin{abstract}
We prove that the tensor product of a simple and a finite dimensional 
$\mathfrak{sl}_n$-module has finite type socle. This is applied to 
reduce classification of simple $\mathfrak{q}(n)$-supermodules to
that of simple $\mathfrak{sl}_n$-modules. Rough structure of 
simple $\mathfrak{q}(n)$-supermodules, considered as 
$\mathfrak{sl}_n$-modules, is described in terms of the
combinatorics of category $\mathcal{O}$.
\end{abstract}

\maketitle


\vspace{5mm}

\section{Introduction}\label{sec1}

For a finite dimensional Lie algebra $\fk$, we consider modules of the form $S\otimes E$, 
where~$S$ is a simple (but not necessarily finite dimensional) $\fk$-module and $E$ 
is a finite dimensional $\fk$-module. The module $S\otimes E$ is always noetherian and
it is natural to ask, see e.g. \cite[Section~1.3]{Kostant}, whether $S\otimes E$ is artinian. 
The latter is true, for example, if $\fk\cong \mathfrak{sl}_2(\mathbb{C})$.
However, in the general case the answer is negative, see \cite[Theorem~4.1]{Stafford}
for~$\fk=\mathfrak{sl}_2(\mC)\oplus\mathfrak{sl}_2(\mC)$. 

In the first half of the present paper, we investigate the following two questions.
\begin{itemize}
\item[Q1:] Does the quotient of $S\otimes E$ by its radical have finite length?
\item[Q2:] Is the socle of $S\otimes E$ is an essential submodule?
\end{itemize}
Note that, since the module $S\otimes E$ is noetherian, its socle must have finite length 
and its radical must be superfluous. We show that question Q1 has an affirmative answer for arbitrary Lie algebras
over arbitrary fields. We also show that Q2 has an affirmative answer 
for reductive Lie superalgebras over $\mC$ of type~$A$. For the latter we apply the theory of
projective functors of \cite{BG} and some specifics about Kazhdan-Lusztig combinatorics in type~$A$.

Our interest in Q2 stems from an application to representation
theory of Lie superalgebras, which occupies the second half of the paper. 
Let $\mathfrak{g}$ be a finite dimensional Lie superalgebra
and $\mathfrak{g}_{\bar{0}}$ be the Lie algebra forming the even part of $\mathfrak{g}$.
A basic problem in the representation theory of $\mathfrak{g}$ is classification of 
simple $\mathfrak{g}$-supermodules. This problem is, most probably, too difficult in the
general case. However, a natural variant of this problem is reduction to classification
of simple $\mathfrak{g}_{\bar{0}}$-modules. In \cite{CM} it was shown that for type~I
Lie superalgebras there is a natural bijection between simple $\mathfrak{g}$-supermodules
and simple $\mathfrak{g}_{\bar{0}}$-supermodules (i.e. pairs consisting of a simple 
$\mathfrak{g}_{\bar{0}}$-module and an element in $\{0,1\}$). Such a nice result seems unrealistic outside type~I.
However, in the present paper we provide a weaker result (see Theorem~\ref{thmclass1}) 
for all finite dimensional classical complex Lie superalgebras for which 
$\mathfrak{g}_{\bar{0}}$ has type~$A$. The crucial ingredient in the proof is the fact that any 
$\mathfrak{g}$-supermodule is a quotient of an induced $\mathfrak{g}_{\bar{0}}$-module. 
And this latter fact follows from the fact that, in case
$\mathfrak{g}_{\bar{0}}$ has type~$A$, question Q2 has an affirmative answer
(see Proposition~\ref{wow} in the first half of the paper).

The paper is organized as follows: Section~\ref{sec2} collects all necessary preliminaries for the
first half of the paper.
Section~\ref{sec3} studies socles and radicals in the biggest possible generality.
Section~\ref{sec4} concentrates on similar questions for semisimple Lie algebras over $\mC$.
Section~\ref{sec5} deals with the very specific case of a sum of copies of $\mathfrak{sl}_2$.
In Section~\ref{sec6} we collected necessary preliminaries about Lie superalgebras.
Section~\ref{sec7} is devoted to classification of simple supermodules for Lie superalgebras
with type~$A$ even part and to the study of rough structure of such supermodules.
Finally, Section~\ref{sec8} describes behavior of Kac induction functor, with respect to socles and radicals,
for Lie superalgebras of type I.

{\bf Acknowledgements.}
The first author is partially supported by Vergstiftelsen. 
The second author is supported by the Australian Research Council.
The third author is partially
supported by the Swedish Research Council and the G{\"o}ran Gustafsson Foundation.

\section{Preliminaries on socles and radicals}\label{sec2}

We denote by $\mathbb{N}$ the set of all non-negative integers and by 
$\mathbb{Z}_{>0}$ the set of all positive integers.

\subsection{Socle, radical and multiplicities}\label{SecSocRad}

Let $R$ be a unital ring and $M$ a left $R$-module. A submodule $N\subset M$ 
is {\em essential} if $K\cap N=0$ implies $K=0$, for all submodules $K\subset M$. 
A submodule $N\subset M$ is {\em superfluous} if $K+ N=M$ implies $K=M$, 
for all submodules $K\subset M$. 

The {\em socle} $\soc(M)$ is defined as the sum of all simple submodules of~$M$,
in case~$M$ has a simple submodule, and as zero otherwise. 
Equivalently, $\soc(M)$ is the maximal semisimple submodule of~$M$. 

\begin{definition}\label{DefSoc}
We will say that an $R$-module~$M$ has {\em finite type socle} if
\begin{enumerate}[$($a$)$]
\item $\soc(M)$ has finite length (and hence is a finite direct sum of simple modules);
\item $\soc(M)$ is an essential submodule of~$M$.
\end{enumerate}
\end{definition}

\begin{remark}\label{RemSoc}${}$
{\rm
\begin{enumerate}[$($i$)$]
\item\label{RemSoc.1} Sometimes modules with finite type socle are 
called {\em finitely cogenerated modules}.  We do not use this terminology 
due to lack of symmetry with finitely generated modules, see Remark~\ref{RemRad}\eqref{RemRad.1}.
\item\label{RemSoc.2} If $\soc(M)$ has finite length, then $M$ has 
finite type socle if and only if all non-zero submodules of~$M$ have non-zero socle.
\item\label{RemSoc.3} If $M$ is noetherian, then $\soc(M)$ has finite length.
\end{enumerate}
}
\end{remark}

The {\em radical} $\rad(M)$ is the intersection of all maximal submodules of $M$.
By convention, we have $\rad(M)=M$ in case~$M$ has no maximal submodules.

\begin{remark}\label{remsocle}
{\em 
Equivalently, $\rad(M)$ is the sum of all superfluous submodules. Indeed,
any superfluous submodule is, clearly, contained in any maximal submodule.
Now, assume that there is $x\in\rad(M)$ which is not contained in the 
sum of all superfluous submodules. Then $Rx$
is not superfluous and hence $Rx+N=M$, for some proper submodule $N$.
By Zorn's lemma, without loss of generality, we may assume that $N$ is maximal (with respect to inclusions) 
among all submodules of $M$ that do not contain $x$. As $x\not\in N$
and $Rx+N=M$, we obtain that $N$ is, in fact, a maximal submodule of $M$. 
This now leads to a contradiction with $x\in\rad(M)\subset N$.
}
\end{remark}

\begin{definition}\label{DefRad}
We will say that an $R$-module~$M$ has {\em finite type radical} if
\begin{enumerate}[$($a$)$]
\item $M/\rad(M)$ has finite length;
\item $\rad(M)$ is a superfluous submodule of~$M$.
\end{enumerate}
\end{definition}

\begin{remark}\label{RemRad}${}${\rm
\begin{enumerate}[$($i$)$]
\item\label{RemRad.1} If $M$ has finite type radical, it follows easily that $M$ is finitely generated. In general, being finitely generated (or even noetherian) is not a sufficient 
condition for $M$ to have finite type radical, see Example~\ref{ExCx}.
\item\label{RemRad.2} If $M/\rad(M)$ has finite length and $M$ is finitely generated, 
then $M$ has finite type radical. This can be quickly reduced to the observation that any finitely generated module $N$ satisfies $\rad(N)\not=N$. Note that it is easy to construct examples of modules that are
not finitely generated but whose quotient over their radical has finite length.
For example, take the direct sum of a simple and an indecomposable injective
module over the polynomial ring in one variable.
\item\label{RemRad.3} If $M$ is noetherian, then $\rad(M)$ is superfluous. This follows easily from Remark~\ref{remsocle}.
\item\label{RemRad.4} If $M/\rad(M)$ has finite length, then $M/\rad(M)$ is semisimple. 
This follows from the canonical morphism $\displaystyle M/\rad(M)\hookrightarrow \Pi_{\mm}M/\mm$, 
where $\mm$ ranges over all maximal submodules of $M$.
In this case we write $\topp(M):=M/\rad(M)$. If $M/\rad(M)$ has infinite length, then
$M/\rad(M)$ does not have to be semisimple, see Example~\ref{ExCx}.
\end{enumerate}
}\end{remark}

\begin{example}\label{ExCx}
{\rm
Consider $R=\mC[x]$, with $M=R$ being the (finitely generated) left regular $R$-module. 
For $\lambda\in\mathbb{C}$, let $(x-\lambda)$ denote the ideal in $\mC[x]$ 
generated by the element $x-\lambda$. Then we have 
$$\rad(M)=\bigcap_{\lambda\in\mC}(x-\lambda)=0.$$
Since~$M/\rad(M)$ does not have finite length, $M$ does not have finite type radical.
}
\end{example}

For an $R$-module~$M$, consider a filtration $F_\bullet M$ of $M$ of length $p\in\mN$:
$$0=F_pM\subset\cdots \subset F_{i+1}M\subset F_iM\subset\cdots\subset F_0M=M.$$
For a short exact sequence 
\begin{equation}\label{sesMMM}
0\to M_1\to M\to M_2\to 0,
\end{equation}
we can define $F_iM_1=M_1\cap F_iM$ and $F_iM_2=(F_iM+M_1)/M_1$. These yield filtrations $F_\bullet M_1$ and $F_\bullet M_2$ of~$M_1$ and $M_2$, and short exact sequences
$$0\to F_i M_1\to F_i M\to F_iM_2\to 0,\quad\mbox{for all~$0\le i\le p$.}$$ 
For a simple $R$-module~$L$, the {\em multiplicity} $[M:L]\in\mN\cup\{\infty\}$ of~$L$ in~$M$ is 
$$[M:L]\;=\;\sup_{F_\bullet}\vert\{i\,|\,F_iM/F_{i+1}M\cong L\}\vert,$$
where~$F_\bullet$ ranges over all {\em finite} filtrations of~$M$. By the above, for a short exact sequence~\eqref{sesMMM}, we have
$$[M:L]\,=\,[M_1:L]+[M_2:L].$$

\begin{lemma}\label{LemSimpTop}
Assume that $R$ is an algebra over $\mC$ which is (at most) countably generated. If an $R$-module $M$ has finite type radical and simple top $S$, then
$$\dim_{\mC}\Hom_R(M,N)\;\le\;[N:S],$$
for all $R$-modules $N$.
\end{lemma}
\begin{proof}
Since $M/\rad (M)$ is simple and $\rad(M)$ superfluous, it follows that every proper submodule of $M$ is contained in $\rad(M)$.
Consequently, we have  the inequality $[M/K:S]\ge 1$, for each proper submodule $K$. It also follows from the above (using Dixmier's version of Schur's lemma) that we have
$$\dim\Hom_R(M,S)=1.$$
Furthermore, since a non-zero morphism $\alpha:M\to N$ yields a submodule $\im(\alpha)$ of~$N$ with $[\im(\alpha):S]\ge 1$, we find
that $[N:S]=0$ implies $\Hom_R(M,N)=0$.

Now assume that $n:=[N:S]<\infty$. By definition, $N$ admits a finite filtration $F_\bullet N$, such that each module $Q_i:=F_iN/F_{i-1}N$ satisfies $[Q_i:S]=0$ or $Q_i\cong S$. The latter option occurs exactly $n$ times. By left exactness of the Hom functor and the above paragraph we thus have
$$\dim\Hom_R(M,N)\,\le\,\sum_i \dim\Hom_R(M,Q_i)\,=\, n,$$
which concludes the proof.
\end{proof}


\subsection{Lie algebras} \label{sec2.2}

Let $\mk$ be a field and $\fk$ a finite dimensional Lie algebra over~$\mk$. 
The universal enveloping algebra of~$\fk$ will be denoted by $U:=U(\fk)$ and the center
of $U$ by $Z=Z(\fk)$. Denote by $\Theta$ the set of central 
characters $\chi:Z(\fk)\to\mC$. We set $\mm_\chi=\ker\chi$, for all~$\chi\in\Theta$.

We denote by $U$-mod the abelian category of all finitely generated left $U$-modules.
For $\chi\in \Theta$, denote by $U\mbox{-mod}_{\chi}$ the full subcategory of~$U$-mod 
consisting of all modules on which the action of~$\mm_\chi$ is locally nilpotent. Set
$$U\mbox{-mod}_Z\;=\;\bigoplus_{\chi\in\Theta}U\mbox{-mod}_{\chi}.$$
For each $\chi\in\Theta$, we denote by $\pr_\chi$ the projection from $U\mbox{-mod}_Z$ 
to $U\mbox{-mod}_{\chi}$ with respect to the decomposition above.

Let $\cF$ denote the category of all finite dimensional $U$-modules. 
We denote the duality $\Hom_{\mk}(-,\mk)$ on $\cF$ by $\ast$, where the left module structure 
is obtained using the anti-automorphism of~$U$ given by $X\mapsto -X$, for $X\in\fk$. 

For a finite dimensional module~$E$, we have the corresponding exact endofunctor $F_{E}=-\otimes E$ of~$U$-mod. The functor $F_E$ is left and right adjoint to~$F_{E^\ast}$, see e.g.~\cite[\S 2.1(d)]{BG}.
In case of a semisimple algebra $\mathfrak{g}$, the direct summands of the restriction to~$U\mbox{-mod}_Z$ of the functors $F_E$ are 
known as {\em projective functors}, see \cite{BG}.

For any $M\in U$-mod, we denote by $\Ann(M)$ the two-sided ideal in $U$ that consists of all elements 
which annihilate every vector in $M$. The arguments in \cite[Kapitel~5]{Jantzen} show that, for any $M_1,M_2\in U$-mod, $E\in \cF$ and $\chi\in\Theta$, we have
\begin{equation}
\label{eqJan}
\Ann(M_1)\subset \Ann(M_2)\;\,\Rightarrow\;\, \Ann\big(\pr_\chi(M_1\otimes E)\big)
\subset \Ann\big(\pr_\chi(M_2\otimes E)\big).
\end{equation}

\subsection{Gelfand-Kirillov dimension and Bernstein number}\label{sec2.3} 

We fix an arbitrary field~$\mk$ and a finite dimensional Lie algebra $\fk$ over~$\mk$. 
Consider the filtration 
$$U=\bigcup_{n\in\mN}U_n,$$ 
where $U_n$ is spanned, as a vector space, by all products of~$n$ 
or fewer elements of~$\fk$. By the PBW theorem, the associated graded algebra $\gr(U)$ is isomorphic to the
the symmetric algebra $S(\fk)$. This allows us to use an alternative definition of 
Gelfand-Kirillov dimension, see~\cite[Section~7]{KrLe}.

For a (non-zero) finitely generated left $U$-module~$M$, with generating subspace~$M_0$, 
set $M_n=U_nM_0$, for all~$n\in\mN$. There exists $n_0\in\mN$, such that we have 
$d\in\mN$ and $\{a_i\in\mZ,0\le i\le d\}$, with $a_d\not=0$, for which
$$\dim_{\mk}M_n\;=\;\sum_{i=0}^d a_i \binom{n}{i},\quad\mbox{for all~$n\ge n_0$}.$$
The right-hand side is the {\em Hilbert-Samuel polynomial} in $n$. The degree of this 
polynomial is the {\em Gelfand-Kirillov dimension} $\GK(M):= d\in\mN$ and the 
leading coefficient is the {\em Bernstein number} $\BN(M):=a_d\in\mZ_{>0}$. 
These two numbers do not depend on the choice of~$M_0$.

By \cite[Lemma~8.8]{Jantzen}, for any $M\in U$-mod and $E\in\cF$, we have
\begin{equation}\label{eqJan-2}
\GK(M\otimes E)\;=\;\GK(M),\quad\mbox{and}\quad \BN(M\otimes E)=\dim_{\mk}(E)\,\BN(M).
\end{equation}
The following statement can be found in \cite[Theorem~7.7]{KrLe}.

\begin{lemma}\label{LemGK}
Let $0\to M_1\to M\to M_2\to 0$ be a short exact sequence in $U$-mod.
\begin{enumerate}[$($i$)$]
\item\label{LemGK.1} We have $\GK(M)=\max\{\GK(M_1),\GK(M_2)\}$.
\item\label{LemGK.2} If $\GK(M_1)=\GK(M)=\GK(M_2)$, then we have $\BN(M)\;=\;\BN(M_1)+\BN(M_2).$
\end{enumerate}
\end{lemma}

\begin{definition}
A finite filtration $F_\bullet M$ of~$M\in U\mbox{{\rm-mod}}$ is {\em GK-complete} if 
$$\GK(F_iM/F_{i+1}M)=\GK(M)\quad\;\mbox{ implies $\;F_iM/F_{i+1}M$ is simple,\hspace{2mm} for all~$i$.}$$
\end{definition}

By the discussion at the end of Section~\ref{SecSocRad}, we have the following observation.

\begin{lemma}\label{LemGKcom}
Assume that $M\in U\mbox{{\rm-mod}}$ admits a GK-complete filtration.
If $N$ is a subquotient of~$M$ with the same Gelfand-Kirillov dimension, 
then $N$ admits a GK-complete filtration as well.
\end{lemma}

\subsection{Triangular decomposition}\label{sec2.4} 

In this section we assume that $\fg$ is a semisimple Lie algebra over~$\mC$. 
In order to recall the classification of projective functors for $\fg$ from \cite{BG}, 
it is convenient to choose a triangular decomposition
$$\fg=\fn^-\oplus\fh\oplus\fn^+$$
of $\fg$ with $\fh$ a Cartan subalgebra and $\fb=\fh\oplus \fn^+$ a Borel subalgebra.
Denote by $\Phi\subset\fh^\ast$ the set of roots of~$\fg$ with respect to~$\fh$. We have 
$\Phi=\Phi^+\sqcup \Phi^-$, with $\Phi^+=-\Phi^-$ being the roots of $\fn^+$. 
For each root $\alpha\in \Phi$, we have the coroot $h_\alpha\in\fh$. 

Consider the Weyl group $W=W(\fg:\fh)$ with its defining action on $\fh^\ast$.  
The group $W$ is generated by $s_\alpha$, 
where $\alpha\in\Phi$, and the action of these generators on $\fh^\ast$ is given by
$s_\alpha(\lambda)=\lambda-\lambda(h_\alpha)\alpha$, where  $\lambda\in \fh^\ast$. 
For each $\lambda\in\fh^\ast$, we denote the stabilizer of $\lambda$ in $W$ 
by $W_\lambda:=\{w\in W\,|\,w(\lambda)=\lambda\}$.  
We have the set of integral weights
$$\Lambda\;:=\;\{\lambda\in\fh^\ast\,\,|\,\, \lambda(h_\alpha)\in\mZ,\;\mbox{ for all~$\alpha\in\Phi$}\}.$$ 
For each $E\in\cF$, we denote by $\supp(E)\subset\Lambda$ the {\em support} of $E$, that is the set of 
all $\fh$-weights of $E$. Then we have 
\begin{displaymath}
\Lambda=\bigcup_{E\in\cF} \supp(E).
\end{displaymath}
For a coset $\Gamma\in\fh^\ast/\Lambda$, we have the 
{\em integral Weyl group} $W^\Gamma\subset W$, which is generated by all reflections 
corresponding to $\alpha\in\Phi$ for which $\lambda(h_\alpha)\in\mZ$, where $\lambda\in\Gamma$.

We have the Harish-Chandra isomorphism $\eta^\ast:Z(\fg)\stackrel{\sim}{\to}S(\fh)^W$ 
and an epimorphism
$$\eta: \fh^\ast\tto \Theta, \quad\mbox{where $\eta(\lambda)({}_-)=\eta^\ast({}_-)(\lambda)$, 
for all~$\lambda\in\fh^\ast$.}$$
The fibers of~$\eta$ are precisely the Weyl group orbits in $\fh^\ast$. 
We call a central character $\chi\in\Theta$ {\em regular} if the set $\eta^{-1}(\chi)$ 
has size $|W|$. We call~$\chi$ {\em integral} if $\eta^{-1}(\chi)\subset\Lambda$.

We introduce the partial order $\le$ on $\fh^\ast$ which is generated by $\mu\le\lambda$, 
if $\mu=s_\alpha\lambda$, for some $\alpha\in\Phi^+$ with $\lambda(h_\alpha)\in\mathbb{Z}_{\geq 0}$.
Hence, $\lambda\in\fh^\ast$ is dominant (maximal) if $\lambda(h_\alpha)\not\in\mZ_{<0}$, 
for all~$\alpha\in\Phi^+$. 

Associated to the triangular decomposition above, we have the BGG category $\cO$
defined as the full subcategory of~$U$-mod consisting of all weight 
modules which are locally $U(\fn^+)$-finite, 
see \cite{BGG76, Hu08}.
We denote by $\rho$ the half of the sum of all elements of~$\Phi^+$. For $\lambda\in\fh^\ast$, 
we have the Verma module $\Delta_\lambda=U\otimes_{U(\fb)}\mC_{\lambda-\rho}$, induced 
from the one-dimensional module $\mC_{\lambda-\rho}$ of~$\fb$ on which $\fh$ acts 
through $\lambda-\rho$. The unique simple quotient of~$\Delta_\lambda$ is denoted 
by $L_\lambda$ and the projective cover of $L_\lambda$  in $\cO$ is denoted by $P_\lambda$.

\subsection{Projective functors for semisimple Lie algebras}\label{sec2.5}

We keep the notation and assumptions of the previous subsection.
Following \cite[Section~1.4]{BG}, we have the set 
$$\Xi^0\,=\,\{(\mu,\lambda)\,|\,\lambda-\mu\in \Lambda\}\;\subset\; \fh^\ast\times\fh^\ast.$$
We set $\Xi=\Xi^0/W$, for the diagonal action of $W$ on $\fh^\ast\times\fh^\ast$. 
Each class in $\Xi$ contains at least one 
$(\mu,\lambda)$, where $\lambda$ is dominant and $\mu$ is such that $\mu\le w\mu$, 
for all~$w\in W_\lambda$. We call such a pair a {\em proper representative}.
The following claim can be found in \cite[Theorem~3.3]{BG}.

\begin{lemma}\label{ClassBG}
Each projective functor on $U\mbox{\rm-mod}_Z$ decomposes into a finite direct sum 
of indecomposable projective functors. We have a bijection $\xi\mapsto F(\xi)$ 
between $\Xi$ and the set of isomorphism classes of indecomposable projective functors as follows. 
For each proper representative $(\mu,\lambda)$ of $\xi\in \Xi$, the functor 
$F^{\mu}_{\lambda}=F(\xi)$ satisfies
$$F^{\mu}_{\lambda}:\; U\mbox{{\rm -mod}}_{\eta(\lambda)}\;\to\; U\mbox{{\rm-mod}}_{\eta(\mu)},\quad \mbox{with }\;F^{\mu}_{\lambda}(\Delta_\lambda)\cong P_{\mu}.$$
\end{lemma}

Note that the above lemma implies that the decomposition into indecomposable projective functors is, in fact, 
unique, up to isomorphism. In case we work with non-integral central characters, we have several ways 
of denoting the same indecomposable projective functor. The first claim of the following lemma can be found
in \cite[Theorem~4.1]{BG} while the second claim can be found in \cite[\S 4.13]{Jantzen}.

\begin{lemma}\label{PropBG}
Consider $\lambda,\mu\in \fh^\ast$ dominant with $\lambda-\mu\in\Lambda$.
\begin{enumerate}[$($i$)$]
\item\label{LemBGequiv}
If $W_\lambda=W_\mu$, then we have an equivalence of categories
$$F^\mu_\lambda:\; U\mbox{{\rm -mod}}_{\eta(\lambda)}\;\stackrel{\sim}{\to}\; U\mbox{{\rm-mod}}_{\eta(\mu)}.$$ 
\item \label{LemWall} 
If $W_\lambda=\{e\}$, then we take the longest element $w_0^\mu$ of~$W_\mu$ and set $\lambda'=w_0^\mu\lambda$. 
The functors $(F^{\lambda'}_\mu,F^\mu_\lambda)$ are biadjoint, moreover, 
$F^\mu_\lambda F^{\lambda'}_\mu\cong \Id^{\oplus |W_\mu|}_{\eta(\mu)}$.
\end{enumerate}
\end{lemma}

Denote by $U\mbox{-mod}^0_{\chi}$ the full subcategory of $U\mbox{-mod}_{\chi}$ consisting 
of all modules $M$ with $\mm_\chi\subset \Ann(M)$. The following claim can be found in \cite[Theorem~3.5]{BG}.

\begin{lemma}\label{LemBGNat}
Let $\lambda\in\fh^\ast$ be dominant and $(\lambda,\mu_1)$ and $(\lambda,\mu_2)$ be proper representatives. 
Evaluation yields an isomorphism 
$$\Nat(F_\lambda^{\mu_1},F_\lambda^{\mu_2})\;\stackrel{\sim}{\to}\; \Hom_{\fg}(F_\lambda^{\mu_1}\Delta_\lambda,F_\lambda^{\mu_2}\Delta_\lambda), \quad\alpha\mapsto \alpha_{\Delta_\lambda},$$
where $\Nat(F_\lambda^{\mu_1},F_\lambda^{\mu_2})$ stands for the space of all natural transformations
from $F_\lambda^{\mu_1}$ to $F_\lambda^{\mu_2}$ as functors 
$U${\rm-mod}$_{\eta(\lambda)}^0\to U${\rm-mod}$_{\eta(\mu_i)}$.
\end{lemma}

\subsection{Kazhdan-Lusztig combinatorics and type~A}\label{TypeA}
We keep the notation and assumptions of the previous subsection.

For a fixed regular dominant $\lambda\in\Lambda$, we use the notation 
$\theta_x:= F^{x\lambda}_\lambda$, for all~$x\in W$. It then follows from the 
validity of the Kazhdan-Lusztig conjecture, see \cite{BB, Kashiwara} and 
Lemma~\ref{ClassBG}, that the composition of the projective functors 
$\theta_x$ is governed by Kazhdan-Lusztig combinatorics. Concretely, 
if we have the {\em Kazhdan-Lusztig basis} $\{C'_w\,|w\in W\}$ of the group 
ring $\mZ W$ ({\it i.e.} the Hecke algebra of $W$specialized at $q=1$) 
of \cite[\S 1]{KL}, then we have
$$C'_xC'_y\;=\;\sum_{z\in W}h_{x,y,z}C'_z\quad\mbox{and}\quad \theta_y\theta_x\;=
\;\bigoplus_{z\in W}\theta_z^{\oplus h_{x,y,z}},$$
for the same coefficients $h_{x,y,z}\in\mathbb{Z}_{\geq 0}$, see \cite[Corollary~5.2.4]{Irving}.

Consequently, the Kazhdan-Lusztig preorders of \cite[\S 1]{KL} can be realised 
by projective functors on a regular block. For convenience, we extend this to 
the entire category of projective functors.
We thus introduce the following preorder $\preceq$ on the set of (isomorphism 
classes of) indecomposable projective functors. We say that $F\preceq G$ if 
$G$ appears as a direct summand of~$F'\circ F\circ F''$, for some projective 
functors $F',F''$. We have the corresponding equivalence relation, denoted 
$F\sim G$, which means that $F\preceq G\preceq F$. Equivalence classes for
$\sim$ are called {\em two-sided cells}. Similarly one defines the
{\em left} and the {\em right preorders} and the {\em left} and the {\em right
cells}.

If the Weyl group $W$ is the symmetric group, the preorder $\preceq$ 
can be described in terms of the dominance order on partitions
using the Robinson-Schensted correspondence, see~\cite[Theorem~5.1]{Geck}. 
As a consequence, each two-sided Kazhdan-Lusztig cell contains the longest element 
of some parabolic subgroup. Another consequence (see~\cite[Corollary~5.6]{Geck}) 
is that a left and a right cell inside the same two-sided cell
intersect in at most one element. By the above and \cite[Theorem 5.3]{Geck}
this translates into two well-known facts {\em for Lie algebras of type~$A$. }

Let $\lambda\in\fh^\ast$ be integral, regular and  dominant, and $\theta$ 
an indecomposable projective endofunctor of $U$-mod$_{\eta(\lambda)}$.
\begin{itemize}
\item[Fact 1:] There exists $x\in W$, which is the longest 
element of a parabolic subgroup such that $\theta\sim \theta_{x}$.
\item[Fact 2:] The only indecomposable projective endofunctor 
$\theta'$ of $U$-mod$_{\eta(\lambda)}$ which satisfies $\theta'\sim\theta$ 
and appears both as a direct summand in $\theta\circ G_1$ and $G_2\circ \theta$, 
for some projective functors $G_1,G_2$, is $\theta$ itself. 
\end{itemize} 
We translate these facts into the formulation that we will require.

\begin{lemma}\label{LemCell}
Assume $\fg$ is of type~$A$ and fix a $\sim$-equivalence class of projective functors.
\begin{enumerate}[$($i$)$] 
\item\label{LemCell.1} The class contains $F^\mu_\mu=\Id_{\eta(\mu)}$, 
the identity functor of~$U${\rm -mod}$_{\eta(\mu)}$, for some dominant $\mu\in\fh^\ast$. 
\item\label{LemCell.2} No other indecomposable projective endofunctors 
of~$U${\rm -mod}$_{\eta(\mu)}$ are contained in the class.
\end{enumerate}
\end{lemma}

\begin{proof}
We start from an arbitrary indecomposable projective functor
$$F: U\mbox{-mod}_{\chi_1}\to U\mbox{-mod}_{\chi_2}$$
and consider its equivalence class. It follows from 
Lemma~\ref{PropBG}\eqref{LemWall} that $F\sim G$, for some 
endofunctor $G$ of $U\mbox{-mod}_{\eta(\lambda)}$, 
where $\lambda$ is any fixed dominant regular weight in $\eta^{-1}(\chi_1)+\Lambda$. 
 
Assume, for simplicity, that $\lambda$ is integral. The non-integral case 
is proved using the same arguments, since the integral Weyl group in type~$A$
is always of type~$A$.
By Fact~1, we have $G\sim\theta_{x_0}$, for the longest element $x_0$ of some 
parabolic subgroup of~$W$. Take some dominant $\mu\in\Lambda$ such that $W_\mu$ 
is this parabolic subgroup. Then we have $\theta_{x_0} =F^{\lambda'}_{\mu}F^{\mu}_{\lambda}$. 
By Lemma~\ref{PropBG}\eqref{LemWall}, we have
 $$\theta_{x_0} \sim F^{\lambda'}_\mu\sim F_\lambda^\mu\sim \Id_{\eta(\mu)}.$$
Consequently, we have $F\sim \Id_{\eta(\mu)}$, which concludes the proof of part 
\eqref{LemCell.1}. 
 
To prove part \eqref{LemCell.2}, we assume we have $\Id_{\eta(\mu)}\sim H$, 
for some indecomposable projective endofunctor $H$ of $U${\rm -mod}$_{\eta(\mu)}$. 
We set $H_1:=F^{\lambda'}_\mu H F^\mu_\lambda$. 
Lemma~\ref{PropBG}\eqref{LemWall} implies that
$$\theta_{x_0}H_1\;=\; H_1^{\oplus |W_\mu|}\;=\; H_1\theta_{x_0}
\quad\mbox{and}\quad F^\mu_\lambda H_1F^{\lambda'}_\mu=H^{\oplus |W_\mu|^2}.$$
The second equation shows that there exists an indecomposable summand $H_1'$ of 
$H_1$ for which we have $H_1'\sim H\sim \theta_{x_0}$. By Fact 2, the first equation 
therefore implies that $H_1'=\theta_{x_0}$. Applying the second equation 
again together with Lemma~\ref{PropBG}\eqref{LemWall} then shows that 
$\Id_{\eta(\mu)}$ must appear as a direct summand in $H^{\oplus |W_\mu|^2}$. 
This is only possible if $H=\Id_{\eta(\mu)}$. This concludes the proof of part \eqref{LemCell.2}.
\end{proof}

\begin{remark}\label{remprof}
{\rm
Facts 1 and 2 generally fail for other Weyl groups, implying that Lemma~\ref{LemCell} 
is specific to Lie algebras of type~$A$. For instance, Fact 1 fails for $B_5$, since in this case 
the unique two-sided cell having the $\mathsf{a}$-value $11$ contains no longest element
of parabolic subgroups (a private communication by Tobias Kildetoft). 
Fact~1 remains valid, for instance, for $\{B_n\,|\, n<5\}$, but there Fact~2 fails, for $n>1$, 
see \cite[Appendix]{MMMZ}.
}
\end{remark}

\section{Arbitrary Lie algebras}\label{sec3}

Fix a finite dimensional Lie algebra $\fk$ over a field~$\mk$ and set $U=U(\fk)$.

\subsection{Connection with GK dimension}\label{sec3.1} 

\begin{theorem}\label{ThmMain}
Fix a simple~$U$-module~$S$ and $E\in \cF$. Then~$T:=S\otimes E$ has finite type radical, 
and $\soc(T)$ has finite length. Furthermore, the following are equivalent:
\begin{enumerate}[$($a$)$]
\item\label{ThmMain.1} The module~$T=S\otimes E$ has finite type socle.
\item\label{ThmMain.2} Any $N\in U$-mod with non-zero morphism $N\otimes E^\ast\to S$ has a simple subquotient~$L$ with non-zero morphism $L\otimes E^\ast\to S$.
\item\label{ThmMain.3} Every non-zero submodule of  $T$ contains a simple subquotient of 
Gelfand-\-Ki\-ri\-llov dimension $\GK(S)$.
\end{enumerate}
\end{theorem}

We start the proof with the following lemma.

\begin{lemma}\label{LemGKses}
With $T$ as in Theorem~\ref{ThmMain}, for any short exact sequence of~$U$-modules
$$0\to M_1\to T\to M_2\to 0,\quad\mbox{with}\quad M_1\not=0\not= M_2,$$
we have $\GK(M_1)=\GK(S)=\GK(M_2)$.
\end{lemma}

\begin{proof}
We prove that $\GK(M_1)=\GK(S)$. The statement for~$M_2$ is proved similarly. 
By Lemma~\ref{LemGK}(i) and equation~\eqref{eqJan-2}, we have $\GK(M_1)\le \GK(S)$. By adjunction
$$0\;\not=\;\Hom_{\fk}(M_1,T)\;\cong\;\Hom_{\fk}(M_1\otimes E^\ast,S).$$
Hence, $S$ is a quotient of~$M_1\otimes E^\ast$.
Lemma~\ref{LemGK}(i) and equation~\eqref{eqJan-2} thus 
imply that $\GK(S)\le \GK(M_1)$, which concludes the proof.
\end{proof}

\begin{lemma}\label{LemRade}
There exists $k\in\mathbb{Z}_{\geq 0}$, such that $T$ cannot have a semisimple submodule or quotient of length greater than $k$.
\end{lemma}

\begin{proof}
We prove the claim for quotients, the case of submodules is proved similarly or follows from Remark~\ref{RemSoc}(iii).
By Lemma~\ref{LemGKses}, a semisimple quotient of~$T$ is a direct sum of simple modules with the GK dimension of each simple equal to~$\GK(S)$. By Lemma~\ref{LemGK}(ii), we can thus choose $k=\BN(T)$.
\end{proof}

\begin{corollary}\label{CorRad}
The module~$T$ has finite type radical.
\end{corollary}

\begin{proof}
Since $T$ is finitely generated, it suffices by Remark~\ref{RemRad}(ii) to show that the module $T/\rad(T)$ has finite length. If it would have infinite length, we could take an arbitrarily large, but finite, direct sum of simple modules as a quotient of~$T$, which is contradicted by Lemma~\ref{LemRade}.
\end{proof}

\begin{proof}[Proof of Theorem~\ref{ThmMain}]
This first claim is Corollary~\ref{CorRad} and Lemma~\ref{LemRade}. Now we prove the equivalence of \eqref{ThmMain.1}, \eqref{ThmMain.2} and \eqref{ThmMain.3}.

First we prove that \eqref{ThmMain.1} implies \eqref{ThmMain.2}. Take $N$ with non-zero $N\otimes E^\ast\to S$ and consider the corresponding non-zero morphism $\alpha:N\to S\otimes E$. The image $\im(\alpha)$ is a submodule of~$S\otimes E$ and thus has a non-zero socle by assumption. We take a simple module~$L$ in that socle, which, by construction, is a subquotient of~$N$. The inclusion $L\hookrightarrow S\otimes E$ yields a non-zero morphism $L\otimes E^\ast\to S$.

Next we note that \eqref{ThmMain.2} trivially implies \eqref{ThmMain.3}.
So, it remains to show that \eqref{ThmMain.3} implies \eqref{ThmMain.1}.  We assume (c) holds and set $d:=\GK(S)$. To obtain a contradiction via Remark~\ref{RemSoc}\eqref{RemSoc.2}, we assume we have a non-zero submodule~$M$ of~$T$ with zero socle. Assume we have $[M:L]\not=0$, for some simple $U$-module~$L$ with $\GK(L)=d$. We thus must have submodules 
$$M_1\subset M_2\subset M,\qquad\mbox{with}\quad M_2/M_1\cong L.$$
Since $M$ has zero socle, we have $M_1\not=0$. By construction, $M_1$ is again a submodule of $T$ with zero socle. 
By Lemma~\ref{LemGKses}, we have $\GK(M_1)=d=\GK(M)$. By Lemma~\ref{LemGK}\eqref{LemGK.2}, we have $\BN(M)>\BN(M_1)>0$. This means that, after repeating the above construction $M\mapsto M_1$ a finite number times, we obtain a non-zero submodule~$N$ of~$M$ (and hence of~$T$) with zero socle, and such that $[N:L]=0$, for all simple $U$-modules with $\GK(L)=d$.
It thus follows that any simple subquotient~$L$ of~$N$ has Gelfand-Kirillov dimension
less than $d$. Hence 
assumption~\eqref{ThmMain.3} yields a contradiction. This completes the proof.
\end{proof}

\begin{corollary}\label{completeCor}
If $T=S\otimes E$ admits a GK-complete filtration, it has finite type socle.
\end{corollary}
\begin{proof}
By Lemmata~\ref{LemGKcom} and~\ref{LemGKses}, any submodule of~$T$ has a GK-complete filtration. Assume that $T$ does not have finite type socle. By Theorem~\ref{ThmMain}, 
this means that $T$ has a submodule~$N$ with $\GK(N)=d$ and such that $[N:L]=0$, 
for all simple modules $L$ with $\GK(L)=d$. Hence $N$ cannot have a GK-complete 
filtration, a contradiction.
\end{proof}

\subsection{Restriction to blocks}\label{sec3.2}

\begin{theorem}\label{ThmBlock}
Let $S$ be a simple $U$-module with central character $\chi\in \Theta$. 
If $F(S)$ has finite type socle, for each projective endofunctor 
$F$ of~$U\mbox{\rm -mod}_{\chi}$, then we have the following:
\begin{enumerate}[$($i$)$]
\item\label{ThmBlock.1} The module $S\otimes E$ has finite type socle, for each $E\in\cF$.
\item\label{ThmBlock.2} The module $S'\otimes E$ has finite type socle, 
for each $E\in \cF$ and simple submodule $S'$ of~$S\otimes V$, for some $V\in \cF$.
\end{enumerate}
\end{theorem}

\begin{proof}
We start by proving Claim~\eqref{ThmBlock.1}. We set $d=\GK(S)$ and use the 
equivalence between \eqref{ThmMain.1} and \eqref{ThmMain.3} in Theorem~\ref{ThmMain}. 
We thus take an arbitrary submodule $N$ of~$S\otimes E$. By adjunction, we have 
a non-zero morphism $N\otimes E^\ast\to S$, which implies $\pr_{\chi}(N\otimes E^\ast)$ 
is not zero. Since $N\otimes E^\ast$ is a submodule of~$S\otimes E\otimes E^\ast$, 
we have an inclusion
$$\pr_{\chi}(N\otimes E^\ast)\;\hookrightarrow \; \pr_{\chi}(S\otimes E\otimes E^\ast).$$
By assumption, the socle of the right-hand side is an essential submodule. 
Hence we can take a simple submodule $L$ in the socle of the right hand side
which is also contained in the left-hand side. Note that $\GK(L)=d$ 
by Lemma~\ref{LemGKses}. 
So $L$ is a submodule of~$N\otimes E^\ast$, which leads through adjunction to 
a non-zero morphism $L\otimes E\to N$. Since $L\otimes E$ has finite type 
radical, there is a simple submodule $L_1$ of~$\topp (L\otimes E)$ such that 
$[N:L_1]\not=0$. By Lemma~\ref{LemGKses}, we have $\GK(L_1)=d$. 
Claim~\eqref{ThmBlock.1} thus follows from Theorem~\ref{ThmMain}.

Now we consider the set-up of claim~\eqref{ThmBlock.2}. By construction, 
$S'\otimes E$ is a submodule of~$S\otimes V\otimes E$. By claim~\eqref{ThmBlock.1}, 
the latter has finite type socle, completing the proof.
\end{proof}

\subsection{Application to Lie superalgebras}\label{sec3.3}

Let $\fs=\fs_{\bar{0}}\oplus \fs_{\bar{1}}$ be a finite dimensional 
Lie superalgebra over $\mk$, see \cite[Chapter~1]{Musson}. 
The universal enveloping algebra $\widetilde{U}=U(\fs)$ of $\fs$
is a $\mZ_2$-graded associative algebra and a finite ring extension of
$U=U(\fs_{\bar{0}})$. 

In the following we use the term {\em simple $\widetilde{U}$-module} 
to denote any of the following two notions
\begin{enumerate}[(I)]
\item\label{aaa1} A simple $\widetilde{U}$-module, without any reference to the $\mZ_2$-grading;
\item\label{aaa2} A $\mZ_2$-graded $\widetilde{U}$-module which has no proper graded submodules.
\end{enumerate}

\begin{proposition}\label{wow}
Assume that $S\otimes E$ has finite type socle, for every simple $U$-mo\-dule $S$ and 
any $E\in \cF$. Then every simple $\widetilde{U}$-module is a quotient of 
a module of the form $\widetilde{U}\otimes_UL$, for some simple $U$-module $L$.
\end{proposition}

\begin{proof}
Consider a simple $\widetilde{U}$-module $K$ in the sense of  \eqref{aaa1}. 
In particular, it is generated by any vector in $K$. 
We denote by $\res$ the restriction functor from $\widetilde{U}$-modules 
to $U$-modules, with left adjoint $\ind$ and right adjoint $\coind$. 
Since $\res(K)$ is finitely generated, it is a noetherian $U$-module. 
Consequently, there is a simple $U$-module $S$ with non-zero morphism 
$\res(K)\to S$. By adjunction, we have an inclusion $K\hookrightarrow \coind\, S$. 
Applying the exact restriction functor gives
$$\res(K)\hookrightarrow S\otimes \Lambda(\fs_{\bar{1}})^\ast.$$
Since the right-hand side has finite type socle, there is a 
simple $U$-module $L$, for which we have an inclusion 
$L\hookrightarrow \res(K)$. Applying adjunction shows that we have a surjection 
$\ind\, L\tto K$.

Consider a simple $\widetilde{U}$-module $K$ in the sense of  \eqref{aaa2}. 
In particular, it is generated by any vector in $K_{\bar{0}}\cup K_{\bar{1}}$. 
We can follow the above procedure and make sure all relevant morphisms 
respect the $\mZ_2$-grading. Any such morphism to or from a simple 
graded $\widetilde{U}$-module will then again automatically be surjective or injective, respectively.
The claim follows.
\end{proof}


\section{Semisimple Lie algebras over $\mC$}\label{sec4}

In this section, we work under the assumptions that $\mk=\mC$ and that the Lie algebra~$\fg$ is semisimple. 

\subsection{Type~A}

\begin{theorem}\label{ThmA}
If $\fg$ is of type~$A$, then every module $S\otimes E$, for $S$ simple and $E$ finite dimensional, has finite type socle.
\end{theorem}

\begin{proof}
Let $S$ be an arbitrary simple module with central character $\chi$. Consider an 
indecomposable projective functor $F$ such that $F(S)\not=0$ and for which every 
indecomposable projective functor $G\succeq F$ with $G(S)\not=0$ satisfies $G\sim F$. 
Note that this is possible since we only have finite chains 
$$
F\preceq F'\preceq F''\preceq\cdots\quad\mbox{ with }\quad F\not\sim F'\not\sim F''\not\sim\cdots.
$$
We consider the equivalence class generated by $F$ and take $\mu\in\fh^\ast$ as in Lemma~\ref{LemCell}. 

Since, by assumption, $F\sim\Id_{\eta(\mu)}$, we have projective functors 
$$
F_1: U\mbox{-mod}_\chi\to U\mbox{-mod}_{\eta(\mu)}\qquad\mbox{and}\qquad
F_2:U\mbox{-mod}_{\eta(\mu)}\to U\mbox{-mod}_{\chi}
$$ 
such that $F$ is a direct summand of~$F_2\circ F_1=F_2\circ\Id_{\eta(\mu)} \circ F_1$. 

Now we have that $F_1S$ is not zero. Since $F_1S$ has finite type radical, we can 
choose a simple quotient $L$ of~$F_1S$. By adjunction, we find that $S$ is 
a submodule of~$G_1L$, for $G_1$ the right adjoint of~$F_1$. 

By Theorem~\ref{ThmBlock}\eqref{ThmBlock.2}, it thus suffices to prove that 
$HL$ has finite type socle, 
for each projective endofunctor $H$ of~$U\mbox{-mod}_{\eta(\mu)}$. We claim that $HL=0$ 
for every indecomposable projective endofunctor $H$ different from the identity, which
would thus complete the proof. Indeed, we even have $HF_1S=0$, since the statement in Lemma~\ref{LemCell}\eqref{LemCell.2},
$$
H\succeq \Id_{\eta(\mu)}\sim F\quad\mbox{ with }\quad H\not\sim \Id_{\eta(\mu)}\sim F
$$ 
implies that $H'\succeq F$ with $H'\not\sim F$, for any indecomposable direct summand 
$H'$ of~$HF_1$.
\end{proof}

\subsection{Low ranks in Type B and C}

\begin{theorem}\label{ThmBC}
For $\fg =B_n=\mathfrak{so}_{2n+1}$ or $\fg=C_n=\mathfrak{sp}_{2n}$, with $1\le n\le 4$, 
every module $S\otimes E$, for $S$ simple with integral central character and $E$ 
finite dimensional, has finite type socle.
\end{theorem}

\begin{proof}
For $n=1$, the Lie algebras are of type~$A$. Now we focus on $n\in\{2,3,4\}$. In these cases, 
Fact 1 of Section~\ref{TypeA} remains true, see \cite[Appendix]{MMMZ}, where elements 
$x_0$ are printed in bold font. Consequently, Lemma~\ref{LemCell}\eqref{LemCell.1} 
remains valid. Note that \cite[Appendix]{MMMZ} only discusses $B_3$ and $B_4$. 
However, since the Weyl groups of $B_n$ and $C_n$ are isomorphic, the latter is also 
included. Furthermore, the case for $B_2$ can be computed immediately by hand. 

Lemma~\ref{LemCell}\eqref{LemCell.2} is no longer true as stated. However, one can calculate 
from \cite[Appendix]{MMMZ} that, in those cases where we have some indecomposable 
projective endofunctor $F$ of $U\mbox{-mod}_{\eta(\mu)}$  different from $\Id_{\eta(\mu)}$ 
but with $F\sim\Id_{\eta(\mu)}$, this $F$ is unique, up to isomorphism. Furthermore, we have 
$F\circ F\cong \Id_{\eta(\mu)}\oplus G$, where $G\not\sim \Id_{\eta(\mu)}$, see the proof
of \cite[Theorem~31]{MMMZ}. 

As in the proof of Theorem~\ref{ThmA} it suffices to prove the following claim. Let 
$L$ be a simple module in $U\mbox{-mod}_{\eta(\mu)}$
which is annihilated by all projective endofunctors which are not in the equivalence class 
of $\Id_{\eta(\mu)}$, then $HL$ has finite type socle for each projective endofunctor 
$H$ of $U\mbox{-mod}_{\eta(\mu)}.$ In those cases where $\Id_{\eta(\mu)}$ is alone in its class 
there is nothing to prove. So we assume that there exists a second functor $F$ in the
class as in the above paragraph. However, in this case~$F^2$ acts as the identity on simple 
modules which are annihilated by indecomposable projective endofunctors of  
$U\mbox{-mod}_{\eta(\mu)}$ not in the class of $\Id_{\eta(\mu)}$ (in fact, $F$
is an equivalence between appropriate subcategories of modules). Consequently $FL$ 
is a simple module.
\end{proof}

\begin{remark}\label{RemBC}
{\rm 
We can use the same arguments as in the proof of Theorem~\ref{ThmBC} 
for semisimple Lie algebras $\fg=\fk\oplus\fl$, with $\fk$ of type $A$ and $\fl$
of type $B_n$ or $C_n$, for $1\le n\le 4$. Note that the arguments 
fail when we consider $\fk\oplus\fl$, with both $\fk$ and $\fl$ of low rank type $B$ and $C$.
}
\end{remark}

\subsection{Regular central character reduction for semisimple Lie algebras}\label{SecRed}
Now we return to arbitrary semisimple Lie algebras over $\mC$.
We will show that for Q2, it will suffice to consider simple 
modules with regular central character.

\begin{theorem}\label{ThmOut}
Take central characters $\chi_r$ and $\chi_s$ with $\chi_r$ regular and $\chi_s$ singular. Take dominant $\lambda\in\eta^{-1}(\chi_r)$ and $\mu \in \eta^{-1}(\chi_s)$ and assume $\lambda-\mu\in\Lambda$. Set $\theta^{on}=F_\lambda^\mu$ and $\theta^{out}=F_\mu^{\lambda'}$, with $\lambda'=w_0^\mu\lambda$, and set $n=|W_\mu|$.

Denote by $\simp_r$ and $\simp_s$ the set of isomorphism classes of simple 
modules with central character $\chi_r$ or $\chi_s$, respectively, and set
$$\simp^s_r\;=\;\{S\in \simp_r\,|\, \theta^{on}S\not=0 \}.$$

\begin{enumerate}[$($i$)$]
\item\label{ThmOut.1} We have a bijection 
$$t:\;\simp_r^s\,\stackrel{\sim}{\to}\,\simp_s, \qquad S \mapsto \theta^{on}S,\quad\mbox{for $S\in \simp_r^s$,}$$
with inverse
$$L\,\mapsto\, \soc(\theta^{out}L)\;\cong\; \topp(\theta^{out}L),\quad\mbox{for $L\in \simp_s$.}$$
\item\label{ThmOut.2} For $L\in\simp_s$, the module $\theta^{out}L$ has simple and finite type socle.
\item\label{ThmOut.3} For $L\in\simp_s$ and $S=t^{-1}(L)=\soc (\theta^{out}L)$, we have $[\theta^{out}L:S]=n.$
\end{enumerate}
\end{theorem}

\begin{remark}
{\rm
Theorem~\ref{ThmOut} remains true if $\chi_r$ is also singular with 
$W_\lambda\subset W_\mu$ if we set $n=|W_\mu|/|W_\lambda|$.
}
\end{remark}

We keep the notation and assumptions of Theorem~\ref{ThmOut} for the rest of the section and start the proof with the following lemma. We freely use Lemma~\ref{PropBG}\eqref{LemWall}, namely biadjunction between $\theta^{on}$ and $\theta^{out}$ and $\theta^{on}\theta^{out}=\Id^{\oplus n}$, freely. 

\begin{lemma}\label{LemTriv0}
For a simple module $L$ with central character $\chi_s$ and a non-zero submodule or quotient $N$ of $\theta^{out}L$, we have $\theta^{on}N\cong L^{\oplus k}$ with $1\le k\le n$.
Moreover, for any simple module $S$, we have
$$\Hom_{\fg}(S,\theta^{out}L)\,\cong\,\Hom_{\fg}(\theta^{out}L,S).$$
\end{lemma}
\begin{proof}
We consider a quotient $N$, the proof for submodules being identical. By exactness of $\theta^{on}$, we find that $\theta^{on}N$ is a quotient of $L^{\oplus n}$.  Since we have a non-zero morphism $\theta^{out}L\to N$, we have a non-zero morphism $L\to \theta^{on}N$, so we find $\theta^{on}N\not=0$. 

By the above, both sides of the proposed isomorphism are zero unless $\theta^{on}S$ is semisimple. In the latter case, we have
$$\Hom_{\fg}(S,\theta^{out}L)\,\cong\,\Hom_{\fg}(\theta^{on}S,L)\,\cong\,\Hom_{\fg}(L,\theta^{on}S)\,\cong\,\Hom_{\fg}(\theta^{out}L,S),$$
which concludes the proof.
\end{proof}

\begin{lemma}\label{lemnew135}
Let $M\in U\text{-}\mathrm{mod}_{\chi_s}^0$ and $\End(\theta^{out})$ be the set of natural 
endotransformations as in Lemma~\ref{LemBGNat}. Then evaluation yields a monomorphism
$$\End(\theta^{out})\,\hookrightarrow\, \End_{\fg}(\theta^{out}M),\quad\alpha\mapsto \alpha_M.$$
If $M$ is simple, then this monomorphism is an isomorphism.
\end{lemma}

\begin{proof}
Take a natural transformation $\alpha:\theta^{out}\Rightarrow \theta^{out}$. 
Then $\theta^{on}(\alpha)$ is a natural 
transformation $\Id^{\oplus n}\Rightarrow \Id^{\oplus n}$. 
By Lemma~\ref{LemBGNat}, we have $\End(\Id^{\oplus n})=\Mat_{n}(\End(\Id))$, 
with $\End(\Id)$ consisting only of scalar multiples of the identity natural 
transformation $\Id\Rightarrow\Id$. Hence a non-zero element in $\End(\Id^{\oplus n})$
evaluated at a non-zero module always yields a non-zero morphism. We claim that $\theta^{on}(\alpha)$ is not zero. 
It then follows that $\theta^{on}(\alpha)$ evaluated at any module $M$ in $U$-mod$_{\chi_s}^0$ 
is not zero, so in particular $\alpha_M$ is not zero. 

To conclude the proof of injectivity, 
we can therefore just observe that we have  $\theta^{on}(\alpha_{\Delta_\mu})\not=0$. 
By Lemma~\ref{LemBGNat}, $\alpha_{\Delta_\mu}:P_{\lambda'}\to P_{\lambda'}$ 
is not zero and $n=[P_{\lambda'}:L_{\lambda'}]$. 
It follows from \cite[4.12(3)]{Jantzen} that $\theta^{on}(\beta)\not=0$, 
for any non-zero endomorphism $\beta$ of~$P_{\lambda'}$.

Now, let $M=L$ be simple. In this case
$$\dim\End_{\fg}(\theta^{out}L)=\dim\Hom_{\fg}(L,L^{\oplus n})=n.$$
Hence, by Lemma~\ref{LemBGNat}  the dimensions of~$\End_{\fg}(\theta^{out}L)$ and $\End(\theta^{out})$ 
agree  and the monomorphism must be an isomorphism.
\end{proof}

\begin{proof}[Alternative proof]
We view $\theta^{out}$ as a functor from $U$-mod$_{\chi_s}^0$ to the full subcategory 
$\cC$ of~$U$-mod$_{\chi_r}$ of modules isomorphic to modules of the form 
$\theta^{out}N$, with $N$ in $U$-mod$_{\chi_s}^0$. By construction, for this 
interpretation of~$\theta^{out}$, $\End(\theta^{out})$ coincides with the algebra of 
natural transformations as in Lemma~\ref{LemBGNat}. Now the restriction of~$\theta^{on}$ 
to $\cC$ has image contained in $U$-mod$_{\chi_s}^0$.
Moreover, it also follows that $(\theta^{on},\theta^{out})$ 
is still a pair of bi-adjoint functors, in the above interpretation.

Now adjunction and evaluation yields a commutative diagram
$$\xymatrix{
\End(\theta^{out})\ar[r]^{\sim}\ar[d]&\Nat(\Id,\Id^{\oplus n})\ar[d]\\
\End_{\fg}(\theta^{out}M)\ar[r]^{\sim}&\Hom_{\fg}(M,M^{\oplus n}).
}$$
By construction, we have to interpret $\Nat(\Id,\Id^{\oplus n})$ as natural transformations with $\Id$ viewed as an endofunctor of $U$-mod$_{\chi_s}^0$. However, this does not differ from the space of natural transformations as in Lemma~\ref{LemBGNat}. Hence, $\dim \Nat(\Id,\Id^{\oplus n})=n$ and the space just consists of linear combinations of identity natural transformations of~$\Id$. In particular, the right vertical arrow is a monomorphism. This forces the left vertical arrow to be a monomorphism as well.

If $M$ is simple, the right vertical arrow is, clearly, an isomorphism.
\end{proof}

\begin{lemma}\label{LemEas}
Fix a simple module $L$ with central character $\chi_s$. Then $\theta^{out}L$ has simple top and socle $S$. Moreover, we have $[\theta^{out}L:S]=n$ and $\theta^{on}S=L$.
\end{lemma}

\begin{proof}
By combining \cite[Theorem~12]{So} and Lemma~\ref{LemBGNat}, it follows that
the algebra $\End(\theta^{out})$ is isomorphic to the coinvariant algebra of
$W_{\mu}$, where $\mu\in \eta^{-1}(\chi_s)$ is dominant. 
In particular, it has simple socle.
By Lemma~\ref{LemTriv0}, the top and socle of $\theta^{out}S$ are isomorphic. If the top and socle are not simple then this clearly contradicts the fact that $\End(\theta^{out})\cong\End_{\fg}(\theta^{out}L)$ 
(see Lemma~\ref{lemnew135}) must have simple socle.

Lemma~\ref{LemSimpTop} implies that
$$n=\dim\Hom_{\fg}(L,L^{\oplus n})=\dim\End_{\fg}(\theta^{out}L)\,\le\, [\theta^{out}L:S].$$
On the other hand, we have $\theta^{on}S=L^{\oplus k}$, for some $k\ge 1$, by Lemma~\ref{LemTriv0}. Applying the exact functor $\theta^{on}$ on $\theta^{out}L$ thus yields
$$n=[L^{\oplus n}:L]\ge k[\theta^{out}L:S].$$
Hence we find $k=1$ and $[\theta^{out}L:S]=n$.
\end{proof}

\begin{proof}[Proof of Theorem~\ref{ThmOut}]
Take some simple module $S\in \simp^s_r$, then $\theta^{on}S$ is non-zero and has finite type radical. In particular, there exists a simple module $L$ with non-zero morphisms $\theta^{on}S\to L$ and $S\to\theta^{out}L$. By Lemma~\ref{LemEas}, we have $t(S)=L$.
Also by Lemma~\ref{LemEas}, the map $t$ is a bijection and parts  \eqref{ThmOut.1} and \eqref{ThmOut.3} follow. 

Now we prove part \eqref{ThmOut.2}. By Theorem~\ref{ThmMain} it is sufficient 
to prove that every submodule of~$\theta^{out}L$ contains a simple subquotient 
of Gelfand-Kirillov dimension $\GK(L)$. For a submodule $N\subset \theta^{out}L$, Lemma~\ref{LemTriv0} implies there exists a non-zero morphism $L\hookrightarrow \theta^{on}N$. By adjunction, there is a non-zero morphism $\theta^{out}L\to N$. Since $\theta^{out}L$ has finite type radical and all simple quotients have Gelfand-Kirillov dimension $\GK(L)$ by Lemma~\ref{LemGKses}, we find that $N$ contains a simple subquotient of Gelfand-Kirillov dimension $\GK(L)$.
\end{proof}

For a simple $U$-module~$S$, we denote by $S\otimes \cF$ the full subcategory of the
category of all $\mathfrak{g}$-modules consisting of all modules isomorphic to the
ones of the form $S\otimes E$, with $E\in\cF$.

\begin{proposition}
If, for every simple module~$S$ with regular central character, all modules in 
$S\otimes \cF$ have finite type socle (resp. a GK-complete filtration), then,
for every simple module~$L$, all modules in 
$L\otimes \cF$ have finite type socle (resp. a GK-complete filtration).
\end{proposition}

\begin{proof}
By Theorem~\ref{ThmOut}, for every simple module~$L$, there exists a simple module~$S$ with regular central character and $V\in\cF$ such that $L$ is a direct summand of~$S\otimes V$.

Hence each module in $L\otimes\cF$ is a direct summand of a module in $S\otimes\cF$. The conclusion for socles follows from observing that any submodule of a module with finite type socle has finite type socle. The conclusion about GK-complete filtrations follows from Lemmata~\ref{LemGKses} and~\ref{LemGKcom}.
\end{proof}

\begin{corollary}
Let $w_0'$ be the longest element of a Coxeter subgroup $W'\subset W$ and $S$ a simple 
module with integral regular central character. If $\theta_{w_0'}S\not=0$, then 
$\theta_xS$ has finite type socle, for all~$x\in W'$.
\end{corollary}

\begin{proof}
We fix some dominant and integral $\mu$ such that $W_\mu=W'$. Then we can write 
$\theta_{w_0'}=\theta^{out}\theta^{on}$ using the notation as in Theorem~\ref{ThmOut}.
From Theorem~\ref{ThmOut} it follows that  $\theta_{w_0'} S$ has finite type socle and 
$S=\soc(\theta_{w_0'} S)$. From Kazhdan-Lusztig combinatorics it follows that $\theta_x\theta_{w_0'}$ 
is a direct sum of copies of~$\theta_{w_0'}$. Hence $\theta_xS$ is a submodule of a 
direct sum of copies of~$\theta_{w_0'}S$ and also has finite type socle.
\end{proof}

\section{Direct sums of copies of~$\mathfrak{sl}_2$}\label{sec5}

In this section, we fix $j\in\mZ_{>0}$ and set $\fg=\mathfrak{sl}_2(\mC)^{\oplus j}$. 
For this $\mathfrak{g}$, we can strengthen Theorem~\ref{ThmA} as follows.

\begin{theorem}\label{ThmSl2}
For $\fg=\mathfrak{sl}_2(\mC)^{\oplus j}$ and any simple $\fg$-module $S$ and $E\in \cF$, 
the module $S\otimes E$ has a GK-complete filtration.
Moreover, if $j=1$, then $S\otimes E$ even has finite length.
\end{theorem}

We label the $j$ copies of~$\mathfrak{sl}_2$ by the index $1\le i\le j$. We denote 
by $\epsilon_i\in\fh^\ast$ the weight half the positive root of the $i$-th copy. Accordingly, we write $\lambda=\sum \lambda_i\epsilon_i$, with 
$\lambda_i\in\mC$, for~$\lambda\in\fh^\ast$.  The Weyl group acts on $\fh^\ast$ by 
changing signs for the $\lambda_i$. For $1\le i\le j$, we denote by $V_i\cong\mC^2$ the natural module for the $i$-th 
copy of~$\mathfrak{sl}_2$, interpreted as a module for~$\fg$ in the obvious way. So, in this way we have 
$\supp(V_i)=\{\epsilon_i,-\epsilon_i\}$. 

Fix $i$ such that $1\le i\le j$. Denote by mod${}_i$ the full subcategory of~$U$-mod 
consisting of all modules $M$ on which the $i$-th copy of~$\mathfrak{sl}_2\subset \fg$ 
acts locally finitely. Clearly, $-\otimes V_i$ restricts to an endofunctor on mod${}_i$. 
We also denote by $\simp^i$ the set of isomorphism classes of simple $\fg$-modules 
which are {\em not} in mod${}_i$.

\begin{lemma}\label{LemSl2int}
Take a simple module $S\in \simp^i$. Then we have $S_1,S_2\in \simp^i$ and a filtration of~$T=S\otimes V_i$
$$0\subset F_2T\subset F_1T\subset T,\qquad\mbox{with}\quad F_2T\cong S_2\mbox{ and }T/F_1T\cong S_1.$$
If $j=1$, so $\fg=\mathfrak{sl}_2$, the module $S\otimes V_i$ has finite length.
\end{lemma}

\begin{proof}
Let $\chi$ be the central character of~$S$. We take a dominant $\lambda\in\eta^{-1}(\chi)$, which means $\lambda_i\not\in\mZ_{<0}$.

Assume first that $\lambda_i\not\in\mZ$. It follows from 
\begin{equation}\label{eqsl2}
\Delta_\lambda\otimes V_i\;\cong\;\Delta_{\lambda+\epsilon_i}\oplus \Delta_{\lambda-\epsilon_i}
\end{equation}
and Lemmata~\ref{ClassBG} and~\ref{PropBG}\eqref{LemBGequiv} that $-\otimes V_i$ 
restricted to $U$-mod${}_\chi$ decomposes into a direct sum of an equivalence with 
$U$-mod${}_{\eta(\lambda+\epsilon_i)}$ and $U$-mod${}_{\eta(\lambda-\epsilon_i)}$. 
Hence $S\otimes V_i$ is a direct sum of two simple modules. 
That both are in $\simp^i$ follows by adjunction.

If $\lambda_i\in\mZ_{>1}$, the argument of the previous paragraph still applies.

If $\lambda=1$, we still have equation~\eqref{eqsl2}. We write $\theta^{on}=F_\lambda^{\lambda-\epsilon_i}$, as in Section~\ref{SecRed}.
By Lemmata~\ref{ClassBG} and \ref{PropBG}\eqref{LemBGequiv}, $S\otimes V_i$
is a direct sum of a simple module and $\theta^{on}S$.
By Theorem~\ref{ThmOut}(i), the latter module is zero or simple. 

Finally, we consider $\lambda_i=0$, which implies
$$\Delta_\lambda\otimes V_i\cong\; P_{\lambda-\epsilon_i}.$$
We write $\theta^{out}=F_\lambda^{\lambda-\epsilon_i}$ as in Section~\ref{SecRed}. Lemma~\ref{ClassBG} implies that $S\otimes V_i\cong \theta^{out}S$.  
That $S\otimes V_i$ is of the desired form then follows from Theorem~\ref{ThmOut}\eqref{ThmOut.1}.

Since $\theta^{on}\theta^{out}=\Id^{\oplus 2}$, it follows that $\theta^{on} F_1T/F_2T=0$.

For $\fg=\mathfrak{sl}_2$, equation~\eqref{eqJan} thus implies that $F_1T/F_2T$ must 
be a direct sum of a number of copies of one fixed simple finite dimensional module. 
Since $T$ is noetherian, it follows that $F_1T/F_2T$ is finitely generated and thus 
finite dimensional.
\end{proof}

\begin{lemma}\label{Lemf}
Let $S$ be a simple $U$-module in {\rm mod}${}_i$. Then $S\otimes V_i$ is a direct sum of simple modules in {\rm mod}${}_i$.
\end{lemma}

\begin{proof}
Let $\chi$ be the central character of~$S$.
The explicit description of the generator of the center in the
universal enveloping algebra of $\mathfrak{sl}_2$ implies 
that any dominant $\lambda\in\eta^{-1}(\chi)$ satisfies $\lambda_i\in\mZ_{>0}$. 
It then follows from the same arguments as in Lemma~\ref{LemSl2int} that 
$S\otimes V_i$ is simple, if $\lambda_i=1$, and is a direct sum of two simple 
modules if $\lambda_i>1$.
\end{proof}

\begin{corollary}\label{CorSl2}
For any simple $\fg$-module $S$ and $1\le i\le j$, 
the module $S\otimes V_i$ has a GK-complete filtration.
\end{corollary}

\begin{proof} 
By Lemma~\ref{Lemf}, we can assume $S\in \simp^i$.

We fix an arbitrary $d\in\mN$.
First we define an equivalence relation on simple $\fg$-modules in $\simp^i$ with 
Gelfand-Kirillov dimension $d$. We let $\sim$ be the minimal equivalence relation 
generated by relation $S\sim S'$ if we have a non-zero morphism 
$S\otimes V_i\to S'$ or $S'\to S\otimes V_i$. Since $V_i$ is self-dual, 
this is indeed an equivalence relation.

Take a simple module $S$ with $\GK(S)=d$ and $\BN(S)$ minimal in its equivalence class. Set $e:=\BN(S)$. By equation~\eqref{eqJan-2}, we have 
$\BN(S\otimes V_i)=2e$. The simple modules $S_1,S_2$ in Lemma~\ref{LemSl2int} thus satisfy $\BN(S_1)+\BN(S_2)\le 2e$ by Lemma~\ref{LemGK}(ii). By minimality of~$e$, we thus have $\BN(S_1)=e=\BN(S_2)$. With notation of Lemma~\ref{LemSl2int}, it follows that
we have $\GK(F_1T/F_2T)<d$. Consequently $S\otimes V_i$ has a GK-complete filtration.

The above paragraph also shows that $S_1$ and $S_2$ are the unique simple modules which appear as submodules or quotients of~$S\otimes V_i$, by Lemma~\ref{LemGKses}. 
Applying the same procedure to $S_1,S_2$ iteratively shows that every simple module in the equivalence class of~$S$ has Bernstein number $\BN(S)$. Hence the condition for $S$ to have minimal $\BN(S)$ in its equivalence class was not actually a restriction. This completes the proof.
\end{proof}

\begin{proof}[Proof of Theorem~\ref{ThmSl2}]
Every finite dimensional $\fg$-module is a direct sum of a module of the form 
$$\bigotimes_{1\le i\le j} V^{\otimes \alpha_i}_i,$$
for suitable $\alpha_i\in\mN$. The statement about GK-complete filtrations thus follows from Corollary~\ref{CorSl2}. The statement for socles then follows from Corollary~\ref{completeCor}.

The claim for~$\mathfrak{sl}_2$ follows from Lemma~\ref{LemSl2int}. 
\end{proof}

\section{Lie superalgebra preliminaries}\label{sec6}

\subsection{Setup}\label{sec6.0}  

In this section, we will introduce the setup of Lie superalgebras.  
We refer to \cite{ChWa,Musson} for more details. From now on we work over the field $\mC$ of complex numbers.

We let $\widetilde{\fg}$ be a finite-dimensional Lie superalgebra with $\mathbb Z_2$-graded decomposition
\begin{displaymath}
\widetilde{\fg}=\widetilde{\fg}_{\bar{0}}\oplus \widetilde{\fg}_{\bar{1}}.
\end{displaymath}
From now on we assume that $\fg:=\widetilde{\fg}_{\bar{0}}$ is a reductive Lie algebra of type~$A$
and that $\widetilde{\fg}_{\bar{1}}$ is a semi-simple $\widetilde{\fg}_{\bar{0}}$-module.
The Weyl group $W$ of $\widetilde{\fg}$ is defined to be the Weyl group of the reductive Lie algebra $\fg$
and we keep notation and terminology of Section \ref{sec2.4}.



Let $V = V_{{\bar {0}}}\oplus V_{\bar{1}}$ be a superspace. For a given homogeneous 
element $v\in V_{i}$, where $i \in \mathbb{Z}_2$, we let $\overline {v}$= $i$ denote its parity. 
We denote  the parity change functor by $\Pi$ on the category of superspaces,
cf. \cite[Section 1.1.1]{ChWa}. For a $\mZ_2$-graded associative algebra $A$, we denote
by $A$-smod the category of all finitely generated $\mZ_2$-graded modules with grading preserving
 homomorphisms. Note that when $A$ is reduced, {\it i.e.} $A=A_{\oa}$, we have
 $$A\mbox{-smod}\;\cong\; A\mbox{-mod}\,\oplus\, \Pi (A\mbox{-mod}).$$

\subsection{Categories of (super)modules}\label{sec6.2}     

We denote  the universal enveloping algebras  of our Lie (super)algebras 
by $\widetilde{U} := U(\wfg)$ and 
$U:=U(\fg)=U(\wfg_{\bar {0}})$. Let $Z(\wfg)$ 
and $Z(\fg)$ denote the center of $\widetilde{U}$ and $U$, 
respectively. Also, the center of $\mathfrak g$ is denoted  by 
$\mathfrak z(\mathfrak g)$.

We set $\wfg$-smod $= \widetilde{U}$-smod and $\fg$-smod 
$= U$-smod. By our assumptions in \ref{sec6.0} a supermodule $M$ over a Lie superalgebra 
is not necessary isomorphic to its parity changed counterpart $\Pi M$. 
  We have the exact restriction, induction and coinduction functors 
\begin{gather*}
\mathrm{Res}:= \mathrm{Res}_{\widetilde{\fg}_{\bar {0}}}^{\wfg}:\wfg \text{-smod}\to 
\wfg_{\bar {0}}\text{-smod},\\
\mathrm{Ind}:=\mathrm{Ind}_{\wfg_{\bar {0}}}^{\wfg}: 
\wfg_{{\bar {0}}}\text{-smod}\rightarrow  \wfg \text{-smod},\\
\mathrm{Coind}:=\mathrm{Coind}_{\wfg_{\bar {0}}}^{\wfg}: 
\wfg_{{\bar {0}}}\text{-smod}\rightarrow  \wfg \text{-smod}.
\end{gather*}
By \cite[Theorem 2.2]{BF} (also see \cite{Go}), the functors
$\mathrm{Ind}$ and $\mathrm{Coind}$ are isomorphic, up to the equivalence given by 
tensoring with the one-dimensional $\fg$-module on the
top degree subspace of $\Lambda \wfg_{{\bar{1}}}$.

By abusing notation we let $\cO$ be the full subcategory of $U$-smod of weight 
modules which are locally $\fb$-finite. In principle, we would
have to write $\cO\oplus\Pi\cO$ for this category, in order to be consistent 
with Section~\ref{sec2.4}. Similarly, we let $L_\la$ denote the simple 
$\fg$-module in unspecified parity. We denote by $\widetilde{\cO}$ the full 
subcategory of $\widetilde{U}$-smod of supermodules which restrict to $\cO$.

For a given $\wfg$-supermodule (resp. $\fg$-supermodule) $X$, 
we denote its $\widetilde U$-annihilator (resp. $U$-annihilator) by $\text{Ann}_{\widetilde U}(X)$ 
(resp.  $\text{Ann}_{U}(X)$). The following theorem is due 
to Duflo in \cite{Duflo}.

\begin{theorem} \label{Dufthm}
Let $V$ be a simple $\fg$-supermodule. 
Then there exist $\lambda \in \mathfrak h^*$ 
such that $\emph{Ann}_{U}(V) = \emph{Ann}_{U}(L_\la)$.
\end{theorem}

We let $\cF$ denote the category of finite dimensional semisimple 
$\fg$-supermodules, and we let~$\widetilde\cF$ denote 
the category of finite-dimensional $\widetilde\fg$-supermodules which restrict to objects 
in~$\cF$. Then $\mathcal F$ and $\widetilde{\cF}$ are exactly
the categories of finite-dimensional weight $\fg$-supermodules and 
$\wfg$-supermodules, respectively.

For a $\wfg$-supermodule $M$, we denote by 
\begin{itemize}
\item $\widetilde{\cF}\otimes M$ the category of 
$\wfg$-supermodules of the form $V\otimes M$, with $V\in\widetilde{\cF}$;
\item $\add(\widetilde\cF\otimes M)$ the category of all supermodules isomorphic to direct summands 
of objects in $\widetilde\cF\otimes M$;
\item $\langle \widetilde\cF\otimes M\rangle$ the abelian category of all supermodules isomorphic to 
subquotients of supermodules in~$\widetilde\cF\otimes M$. 
\end{itemize}
 
We let~$\mathrm{Coker}(\widetilde\cF\otimes M)$ denote the {\em coker-category} of $M$, that is the 
full subcategory of the 
category of all $\wfg$-supermodules, which consists of all modules $N$ that have a presentation 
$$X\to Y \to N\to 0,$$ where 
$X,Y\in \add(\widetilde\cF\otimes M)$. 
Similarly we define analogous full subcategories of 
$\fg$-supermodules, cf \cite{MaSt08}.




\subsection{Harish-Chandra bimodules}\label{SeHCB}
Following \cite[\S 3.2]{CC}, we will consider a type of Harish-Chandra bimodules where the left action is by a Lie superalgebra
and the right action by the underlying Lie algebra. The corresponding category will be an essential tool in our study of the rough
structure.

 We write $\widetilde{U}$-smod-$U$ for $\widetilde{U}\otimes U^{\op}$-smod.
For a bimodule $Z$ in this category, the $\fg$-module $Z^{\mathrm{ad}}$ 
is the restriction of $Z$
to the adjoint action of $\fg$. This is the restriction via 
$$U\hookrightarrow \widetilde{U}\otimes U^{\op},\qquad X\mapsto X\otimes 1-1\otimes X,\mbox{ for all $X\in\fg$.}$$

Let $\cB$ denote the full subcategory of $\widetilde{U}$-smod-$U$ of bimodules $N$ 
for which $N^{\mathrm{ad}}$ is a direct sum of modules in~$\cF$ with finite multiplicities.  
For a two-sided ideal $J\subset U$, we let $\cB(J)$ denote the full 
subcategory of~$\cB$ of bimodules $X$ such that $XJ=0$. 
We have the functor
$$\cL(-,-):\,  (U\mbox{-smod})^{\op}\times \widetilde{U}\mbox{-smod}\to \cB,$$ where
$\cL(M,N)$ is the maximal submodule of $\Hom_{\mC}(M,N)$ 
which belongs to $\cB$.

By slight abuse of notation, we will use the same notation $\cL$ for the corresponding functor applied to the case $\widetilde{\fg}=\fg$.
Let $M$ be a $U$-module, then the $\fg$-action on $M$ defines a bimodule homomorphism from $U$ to $\mathcal L(M,M)$ with kernel $\text{Ann}_{U}(M)$.
Hence we have a canonical monomorphism
\begin{equation}\label{KSProb}
U/\Ann_{U}(M)\hookrightarrow \cL(M,M). 
\end{equation}
One says that {\em Kostant's problem} for $M$ has a positive solution if \eqref{KSProb} is 
an isomorphism, see \cite{Jo80,Go02,MaMe12}.  Then we have the following variation of~\cite[Theorem~3.1]{MSo} 
established in \cite[Theorem~3.1]{CC}. 

\begin{theorem}\label{BIEquiv}
Consider $M\in U\mbox{{\rm-mod}}$ with central character and set $I:=\Ann_{U}(M)$. 
If the Kostant problem for $M$ has a positive solution and $M$ is projective in 
$\langle \cF\otimes M\rangle$, then 
$$-\otimes_{U}M\,:\;\;\; \mathcal B(I)\,\to\, \mathrm{Coker}(\widetilde{\cF}\otimes \mathrm{Ind}(M))$$
is an equivalence of categories with inverse $\mathcal L(M,-)$.
\end{theorem}

For $E \in\widetilde{\cF}$, we recall that $E\otimes \widetilde{ U}$ is equipped with a natural 
$\wfg$-bimodule structure as in \cite[Section 2.2]{BG} and \cite[Section 2.4]{Co16}: 
\[ X(v\otimes u)Y = (Xv)\otimes (uY) +(-1)^{\overline X\cdot \overline v}v\otimes (XuY),\]
for all homogeneous $X,Y\in \wfg$, $v\in E$ and $u \in \widetilde{U}$.

\section{Rough structure of simple $\widetilde{\fg}$-supermodules}\label{sec7}

\subsection{Motivation}\label{sec7.1}

Two fundamental problems in  representation theory of a group or a ring are:
\begin{itemize}
\item classification of simple modules;
\item understanding how all modules are constructed from simple modules.
\end{itemize}
A natural subproblem of the second problem is determination of multiplicities of
simple subquotients in a given module. In the general case, certain multiplicities might be infinite
or depend on the choice of filtration. 

The paper \cite{KhMa04} studied the structure of a certain class of modules over Lie algebras, called
generalized Verma modules. These modules are obtained by parabolic induction from simple modules over
``smaller'' Lie algebras. It turned out that, given a generalized Verma module, there is a natural
class of simple modules (defined using a certain comparability of annihilators) whose multiplicities 
in the generalized Verma module are well-defined, finite and computable using Kazhdan-Lusztig combinatorics. 
This was called the {\em rough structure} of generalized Verma modules. The most general, to date, 
result on the rough structure of generalized Verma modules was obtained in \cite[Section~11.8]{MaSt08}.

In the study of Lie superalgebras, it is natural to ask about the rough structure of simple
supermodule as modules over the even part of the Lie superalgebras. In \cite{CM},
for classical Lie superalgebras of type $I$, it was shown that
\begin{itemize}
\item the classification of simple supermodules can be reduced to the classification of simple modules
over the even part of the superalgebra;
\item the rough structure of a simple
supermodule as a module over the even part of the Lie superalgebra can be described in terms
of combinatorics of category $\mathcal{O}$.
\end{itemize}

The goal of this section is to address these problems for Lie superalgebra not necessarily of type I, but with underlying Lie algebra of type $A$.

\subsection{Coker categories for induced modules} \label{SeCok} 

With Theorem \ref{BIEquiv} as the main tool, this subsection shall proceed with the study of 
coker-categories of induced modules along the lines of \cite[Section 11.6]{MaSt08}. 

For simplicity, we will work with regular integral central characters by following \cite[Remark~76]{MaSt08}.
The general case then follows by standard techniques, in particular using translations
out and onto the walls and the equivalences from \cite{ChMaWa13}.

We start with reviewing and adapting the setup of \cite[Section~11.6]{MaSt08} to the present paper. 
Note that we work in the generality when $\fg$ is reductive. Each central element of  $\fg$ acts as
a scalar on any simple $\fg$-module, therefore we will work only with $\fg$-modules on which 
the action of the center of $\fg$ is semi-simple. We assume that the center of $\fg$ belongs to
any Cartan subalgebra and we call a weight {\em integral} if it appears in some simple finite
dimensional $\fg$-module.

Let $V$ be a simple $\fg$-supermodule which admits a regular and 
integral central character. 
By Theorem~\ref{Dufthm},  there is 
a dominant weight $\nu$ and an element $\sigma\in W$ such that 
$\text{Ann}_{U} (L) = \text{Ann}_{U}(L_{\sigma\nu})$. 
We may assume that $\sigma$ is contained in a right cell associated with a parabolic 
subalgebra $\mathfrak p \subseteq \fg$ as in \cite[Remark 14]{MaSt08}. Therefore there 
is a dominant weight $\mu$ such that the parabolic block $\mathcal O^{\mathfrak p}_{\mu}$ 
contains exactly one simple module $L_{y \mu}$ and this module is projective, 
see e.g. \cite[3.1]{IrSh88}. As $y \mu=yx \mu$, for any $x$ in the stabilizer
of $\mu$, without loss of generality, we may assume that $ys<y$, for all 
simple reflection $s$ with $s \mu=\mu$. With this assumption, we have that
$L_{y \mu}$ is the translation of $L_{y 0}$ to the $\mu$-wall.

Tensoring, if necessary, with finite dimensional modules,
without loss of generality we may assume that 
$\mu$ is  {\em generic} in the sense of \cite[Subsection~5.3]{MaMe12}.  
Let $F$ be the projective functor given in 
\cite[Proposition 61]{MaSt08} and define $\overline N$ to be the simple quotient of $FL$. We refer the reader to 
\cite[Section~11]{MaSt08} for more details of our setup. In particular, we have that 
\begin{align} \label{SameAnn}
&I:=\text{Ann}_{U}(\overline N) = \text{Ann}_{U}( L_{y \mu}).
\end{align}

The following theorem, which is  \cite[Theorem 66]{MaSt08}, is our main tool to study 
the rough structure for Lie algebras of type~$A$. 

\begin{theorem}\label{thmrsms}  
Let $\mathcal F$ denote the category of 
finite-dimensional weight $\fg$-su\-per\-mo\-du\-les. Then the functor
$$\Xi: = \mathcal L(\overline N, -)\otimes_{U} L_{y\mu}: \mathrm{Coker}(\mathcal F \otimes \overline N) \rightarrow \mathrm{Coker}(\mathcal F \otimes L_{y \mu})$$ 
is an equivalence of categories.
\end{theorem}

By \cite[Lemma 67]{MaSt08}, categories $\mathrm{Coker}(\mathcal F \otimes \overline N)$ and 
$\mathrm{Coker}(\mathcal F \otimes  L_{y \mu})$ are both admissible in the sense of \cite[Section~6.3]{MaSt08}. 

The following lemma combines \cite[Lemma~63]{MaSt08} and \cite[Proposition~65(ii)]{MaSt08}.

\begin{lemma}\label{prj}
The objects  $L_{y \mu}$ and $\overline N$ are projective in  
$\langle \cF\otimes L_{y \mu} \rangle$ and  
$\langle \cF\otimes \overline N\rangle$, respectively. 
\end{lemma}

Now we can formulate the following equivalence of coker-categories.

\begin{corollary}  \label{CBIEquiv}
The functor
$$\widetilde{\Xi}: = \mathcal L(\overline N, -)\otimes_{U} L_{y \mu}: \mathrm{Coker}(\widetilde{\cF} \otimes \mathrm{Ind} (\overline N)) \rightarrow \mathrm{Coker}(\widetilde{\cF} \otimes \mathrm{Ind} (L_{y \mu}))$$ 
is an equivalence sending $\mathrm{Ind}(\overline N)$ to $\mathrm{Ind}( L_{y \mu})$.
\end{corollary}

\begin{proof} 
With $I$ from \ref{SameAnn}, we claim that 
\begin{equation}
\mathrm{Coker}(\widetilde{\cF} \otimes \mathrm{Ind} (\overline N)) 
\xrightarrow{\mathcal L(\overline N, -)} \mathcal B(I) \xrightarrow{-\otimes_{U} L_{y \mu}} \mathrm{Coker}(\widetilde{\cF} \otimes \mathrm{Ind}( L_{y \mu}))
\end{equation} 
are equivalences. Indeed, recall that Kostant's problem has a positive solution for $L_{y \mu}$ and $\overline N$ 
by \cite[Theorem~60(iii)]{MaSt08} and \cite[Proposition 65(iii)]{MaSt08}, respectively. Now 
Lemma~\ref{prj} and Theorem~\ref{BIEquiv} imply the claim. 
	
Finally, since Kostant's problem has a positive solution for $\overline N$, it follows that
\begin{align*}
&\mathcal L(\overline N, \mathrm{Ind} (\overline N))\otimes_{U} L_{y \mu} 
\cong \widetilde{U}\otimes_{U}\mathcal L(\overline N, \overline N)
\otimes_{U} L_{y \mu} \cong \mathrm{Ind}( L_{y \mu}), 
\end{align*}
as $\widetilde{U}$-modules. This completes the proof.
\end{proof}

\subsection{Rough structure of induced $\widetilde{\fg}$-supermodules}\label{sec7.5} 

In this subsection, we obtain a description of the rough structure 
of induced supermodules. 



\begin{lemma}\label{cor7.5-1}
The equivalence $\widetilde\Xi$ in Corollary \ref{CBIEquiv} gives rise to a 
bijection between the sets of isomorphism classes of objects in 
the categories $\add(\widetilde{\cF}\otimes \mathrm{Ind}(\overline N))$ 
and $\add(\widetilde{\cF}\otimes \mathrm{Ind} (L_{y \mu}))$. These two categories are, 
respectively, the categories of projective objects in 
$\mathrm{Coker}(\widetilde{\cF}\otimes \mathrm{Ind}(\overline N))$ and
$\mathrm{Coker}(\widetilde{\cF}\otimes \mathrm{Ind}(L_{y \mu}))$. 
\end{lemma}

\begin{proof} 
For a given $E\in\widetilde{\cF}$,  we have 
\begin{equation}\label{Xicorr}
\begin{array}{rcl}
\widetilde{\Xi}(E \otimes \mathrm{Ind}( \overline N)) 
&\cong &
\mathcal L(\overline N,  \mathrm{Ind}(\mathrm{Res} (E)\otimes \overline N))\otimes_{U} L_{y \mu}\\ 
&\cong & \widetilde{U} \otimes_{U} (\mathrm{Res} (E) \otimes U \otimes_{U}
U/I) \otimes_{U} L_{y \mu}\\ 
&\cong &\mathrm{Ind}(\mathrm{Res} (E)\otimes L_{y \mu})\\ 
&\cong &E\otimes \mathrm{Ind}( L_{y \mu}).
\end{array}
\end{equation} 
Here the first and the last isomorphisms use \cite[Proposition 6.5]{Knapp},
which states that
$$E\otimes\mathrm{Ind}(X)\cong \mathrm{Ind}(\mathrm{Res}(E)\otimes X),$$
for any $\fg$-module $X$,
the second isomorphism uses \cite[Section~6.8]{Jantzen} and the isomorphism
$\mathcal L(\overline N, \overline N)\cong U/I$
and the third isomorphism uses \eqref{SameAnn}.

That $\add(\widetilde{\cF}\otimes \widetilde{U}/\widetilde{U}I)$ 
is the full subcategory of projective modules in $\mathcal B(I)$ is shown in \cite[Proof of Theorem~3.1]{CC}. The claim follows. 
\end{proof}

As a consequence, see \cite[Section~11]{MaSt08}, $\widetilde{\Xi}$ induces a bijection
\begin{displaymath} 
{\mathbf{\widetilde\Xi}}:~\mathrm{Irr}^{\wfg}( \mathrm{Coker}(\widetilde{\cF} \otimes \mathrm{Ind}(\overline N)))\rightarrow
\mathrm{Irr}^{\wfg}(\mathrm{Coker}(\widetilde{\cF} \otimes \mathrm{Ind} (L_{y \mu}))),
\end{displaymath}
between the sets of isomorphism classes of simple $\wfg$-supermodule quotients
of simple objects in $ \mathrm{Coker}(\widetilde{\cF} \otimes \mathrm{Ind}( \overline N))$ and
$\mathrm{Coker}(\widetilde{\cF} \otimes \mathrm{Ind} (L_{y \mu}))$. 

For  $S \in \mathrm{Irr}^{\wfg}( \mathrm{Coker}(\widetilde{\cF} \otimes \mathrm{Ind}( \overline N)))$, 
we define
\begin{displaymath} 
\widetilde{L}_{S}:= \mathbf{\widetilde\Xi}(S)\,\in\widetilde\cO. 
\end{displaymath} 

We are now in a position to state the first main result of this section which describes 
rough structure of induced modules in terms of category $\widetilde{\mathcal{O}}$ combinatorics.

\begin{theorem}[Rough structure of induced modules] \label{thmrough}
For any module $E\in\widetilde{\cF}$ and any module
$S \in \emph{Irr}^{\wfg}( \mathrm{Coker}(\widetilde{\cF} \otimes \mathrm{Ind}( \overline N)))$, 
we have 
\begin{equation}\label{ERStr}
[\mathrm{Ind}(\mathrm{Res} (E)\otimes \overline N):S] =[\mathrm{Ind}(\mathrm{Res} (E)\otimes L_{y \mu}):\widetilde{L}_{S}].
\end{equation}
\end{theorem}

\begin{proof} The proof of equation \eqref{ERStr} follows from a similar argument  using $\Xi$ as given in \cite[Theorem~72]{MaSt08}. To see this, 
let $\widetilde{Q}_S$ and $\widetilde{Q}_{L}$ denote the indecomposable  projective objects in the categories
$\mathrm{Coker}(\widetilde{\cF}\otimes \mathrm{Ind} (\overline N))$ and 
$\mathrm{Coker}(\widetilde{\cF}\otimes \mathrm{Ind}( L_{y \mu}))$ with tops  
$S$ and $\widetilde{L}_{S}$, respectively. By \eqref{Xicorr} we have
\begin{displaymath}
\begin{array}{rcl}
[\mathrm{Ind}(\mathrm{Res} (E)\otimes \overline N): S]
&=& \dim\mathrm{Hom}_U(\widetilde{Q}_S, \mathrm{Ind}(\mathrm{Res} (E)\otimes \overline N))\\
&=& \dim\mathrm{Hom}_U(\widetilde{Q}_{L}, \mathrm{Ind}(\mathrm{Res} (E)\otimes  L_{y \mu}))\\ 
&=& [\mathrm{Ind}(\mathrm{Res} (E)\otimes  L_{y \mu}): \widetilde{L}_{S}].
\end{array}
\end{displaymath}
The claim follows.
\end{proof}

\subsection{Rough structure of simple $\widetilde{\fg}$-supermodules}\label{sec7.6}  

In this subsection, we give a description of the rough structure 
of restrictions to $\fg$ of simple $\wfg$-supermodules in 
$\mathrm{Irr}^{\wfg}(\mathrm{Coker}(\widetilde{\cF} \otimes \mathrm{Ind}( \overline N)))$.  

\begin{lemma} \label{nlem25}
There is an isomorphism $\emph{Res}\circ \widetilde{\Xi} \cong \Xi \circ \emph{Res}$ of functors.
\end{lemma}

\begin{proof}
This follows directly from the definitions. 
\end{proof}

We have a bijection 
$$\mathbf{\Xi} : \mathrm{Irr}^{\fg}(\mathrm{Coker}(\mathcal F \otimes \overline N))\rightarrow \mathrm{Irr}^{\fg}(\mathrm{Coker}(\mathcal F \otimes L_{y \mu})),$$
induced by  $\Xi$,  between the sets of isomorphism classes  of 
the simple $\fg$-quotients of simple objects in $\mathrm{Coker}(\mathcal F \otimes \overline N)$ and in $\mathrm{Coker}(\mathcal F \otimes L_{y \mu})$. 
For a given
\begin{displaymath}
W \in \mathrm{Irr}^{\fg}(\mathrm{Coker}(\mathcal F \otimes \overline N)),
\end{displaymath}
we define  
the related weight $\zeta_W\in \mathfrak h^*$ by 
$L_{\zeta_W} \cong \mathbf\Xi(W)$.
The next statement describes the $\fg$-rough structure of simple 
$\wfg$-supermodules in terms of the combinatorics of category $\mathcal{O}$.

\begin{theorem}[Rough structure of simple $\wfg$-supermodules]
For any $\wfg$-supermodule  
$S\in \mathrm{Irr}^{\wfg}(\mathrm{Coker}(\widetilde{\cF} \otimes \mathrm{Ind} \overline N))$ and
any $\fg$-supermodule
$W\in \mathrm{Irr}^{\fg}(\mathrm{Coker}(\mathcal F \otimes \overline N))$, 
we have
\begin{align}\label{Eq::RoughStr-s}
[\mathrm{Res}(S):W] = [\mathrm{Res}(\widetilde{L}_{S}):L_{\zeta_W}].
\end{align}
\end{theorem}

\begin{proof}
Equality~\eqref{Eq::RoughStr-s} is obtained by applying $\mathbf{\Xi}$ to the 
left hand side and then using Lemma~\ref{nlem25}, cf. \cite[Theorem 73]{MaSt08}.
\end{proof}

\subsection{Classification of simple $\widetilde{\fg}$-supermodules} \label{sec7.4}

Theorem~\ref{thmrsms} reduces the problem of classification of all simple $\fg$-modules
to the following problem. Let $\mathcal{I}$ denote the set of all primitive ideals in $U$
which appear as annihilators of modules over the form $\overline N$, for all possible weight lattices.
For each fixed simple $\fg$-module $S$ with annihilator $I\in \mathcal{I}$
(which we view as some $\overline N$), we have the corresponding bijection
$$
\mathbf\Xi(S) : 
\mathrm{Irr}^{\fg}(\mathrm{Coker}(\mathcal F \otimes S))\rightarrow \mathrm{Irr}^{\fg}(\mathrm{Coker}(\mathcal F \otimes L_{y \mu})).
$$
For $I\in \mathcal{I}$, let $\mathrm{Irr}^{\fg}(I)$ denote the set of 
isomorphism classes of simple $\fg$-modules whose annihilator is $I$.
Consequently, we have:

\begin{corollary}\label{thmclass0}
The set 
\begin{displaymath}
\coprod_{I\in \mathcal{I }}\,\,\,\coprod_{S\in \mathrm{Irr}^{\fg}(I)} 
\mathbf\Xi(S)^{-1}
\big(\mathrm{Irr}^{\fg}(\mathrm{Coker}(\mathcal F \otimes L_{y \mu}))\big)
\end{displaymath}
coincides with  the set of isomorphism classes of simple  $\fg$-modules.
\end{corollary}

This reduces, in some sense, the problem of classification of all $\fg$-modules
to that of classification of all $\fg$-modules with annihilators in $\mathcal{I}$.
For Lie algebras, this might look unimpressive. However, a similar reasoning applied to 
Corollary~\ref{CBIEquiv} gives the following very surprising extension of this to $\widetilde{\fg}$
which, in the same sense, reduces the problem of classification of all $\widetilde{\fg}$-supermodules
to that of classification of all $\fg$-modules with annihilators in $\mathcal{I}$.
Here, again, for each fixed simple $\fg$-module $S$ with annihilator $I\in\mathcal{I}$,
we have the bijection
\begin{displaymath}
\mathbf{\widetilde{\Xi}}(S):\mathrm{Irr}^{\wfg}( \mathrm{Coker}(\widetilde{\mathcal F} \otimes \mathrm{Ind}(S)))
\rightarrow
\mathrm{Irr}^{\wfg}(\mathrm{Coker}(\widetilde{\mathcal F} \otimes \mathrm{Ind} (L_{y \mu}))).
\end{displaymath}

\begin{theorem}[Classification of simple supermodules]\label{thmclass1}
The set 
\begin{equation}\label{abs431}
\coprod_{I\in \mathcal{I }}\,\,\,\coprod_{S\in \mathrm{Irr}^{\fg}(I)} 
\mathbf{\widetilde\Xi}(S)^{-1}
\big(\mathrm{Irr}^{\wfg}(\mathrm{Coker}(\widetilde{\cF} \otimes \mathrm{Ind}(L_{y \mu})))\big)
\end{equation}
coincides with  the set of isomorphism classes of simple $\widetilde{\fg}$-supermodules. 
\end{theorem}

\begin{proof}
As $\fg$ is of type~$A$, we can combine Theorem~\ref{ThmA} with Proposition~\ref{wow}
to conclude, using Corollary~\ref{CBIEquiv}, that every simple $\widetilde{\fg}$-supermodule 
appears in the right hand side of \eqref{abs431}. The claim follows.  
\end{proof}

Although Theorem~\ref{thmclass1} is not as nice as \cite[Theorem~4.1]{CM} (the latter result reduces classification of
simple supermodules over basic classical Lie superalgebras of type I to classification of simple modules
over the corresponding even part Lie algebra), it is fairly clear that \cite[Theorem~4.1]{CM} does not extend to,
for example, Q-type Lie superalgebras given in Subsection~\ref{SectEx} in any easy way.
At the same time, the set 
\begin{displaymath}
 \mathrm{Irr}^{\wfg}(\mathrm{Coker}(\widetilde{\mathcal F} \otimes \mathrm{Ind}(L_{y\mu})))
\end{displaymath}
for  Q-type Lie superalgebras can be described using Subsection~\ref{SectEx}.
Therefore for  Q-type Lie superalgebras, Theorem~\ref{thmclass1} provides significant
progress in classification of simple supermodules.

\subsection{Applications}\label{SectEx} 
We give a short overview of Lie superalgebras to which our results are applicable and cannot be dealt with using the theory of \cite{CM}. We focus on $D(2,1;\alpha), F(4), Q(n)$ and generalized Takiff superalgebras.

First we quickly review the classification of simple modules in $\widetilde{\cO}$.
Let $\widetilde{\fg}$ be a classical Lie superalgebra and choose a Cartan subalgebra~$\fa$ of the derived Lie algebra $\fg'=[\fg,\fg]$ of the underlying Lie algebra. Consider a Weyl group invariant bilinear form $\langle\cdot,\cdot\rangle$ on $\fa^\ast$. Following~\cite[\S 2.4]{Ma}, we choose a $\omega\in \mR\Phi$, where $\Phi$ is the set of roots of $\fg'$, which is generic in the sense that $\langle\omega,\alpha\rangle\not=0$ for all $\alpha\in\Phi$. Then we define a triangular decomposition
$$\widetilde{\fg}\;=\;\widetilde{\fn}^-\oplus\widetilde{\fh}\oplus\widetilde{\fn}^+,$$
where $\widetilde{\fh}$ is the direct sum of $\fa$-weight spaces corresponding to weights $\lambda$ satisfying $\langle\lambda,\omega\rangle=0$ and $\widetilde{\fn}^{\pm}$ corresponds to weights $\lambda$ with $\pm \langle\lambda,\omega\rangle>0$. Clearly $\fh:=\widetilde{\fh}_{\oa}$ and $\fn^{\pm}:=\widetilde{\fn}^{\pm}_{\oa}$ give a triangular decomposition of $\fg$.

Assume now that $[\widetilde{\fh}_{\ob},\widetilde{\fh}_{\ob}]=0$. Then simple $\widetilde{\fh}$-modules are the one dimensional $\fh$-modules with trivial $\widetilde{\fh}_{\ob}$-action $\mC_\lambda$, with $\lambda\in\fh^\ast$. We have the corresponding Verma module
$$\widetilde{\Delta}_\lambda:=U(\widetilde{\fg})\otimes_{U(\widetilde\fb)}\mC_\lambda,$$
with $\widetilde{\fb}=\widetilde{\fh}\oplus\widetilde{\fn}^+$ acting on $\mC_\lambda$ with trivial $\widetilde{\fn}^+$-action. It follows as in the classical case that $\widetilde{\Delta}_\lambda$ has a unique maximal submodule. We denote the quotient by $\widetilde{L}_\lambda$. The following lemma is well-known, see e.g. \cite[Proposition~2]{Ma}.
\begin{lemma}\label{LemSimpO}
The set  $\{\widetilde{L}_\la, \Pi \widetilde{L}_\la|~\lambda \in \mathfrak h^*\}$
	is a complete and irredundant set of representatives of isomorphism classes of simple objects in $\mathcal{O}$.
\end{lemma}
The one parameter family of simple Lie superalgebras $D(2,1|\alpha)$, see \cite[Chapter 4]{Musson}, satisfies the above assumptions and has underlying Lie algebra of type $A$.
 The multiplicities in ${\rm Res}(\widetilde{L}_\lambda)$ for $D(2,1|\alpha)$ have been determined in \cite{ChWa02}.
Therefore our results provide a concrete approach to the rough structure problem. The simple Lie superalgebra $F(4)$, see \cite[Chapter 4]{Musson}, has underlying Lie algebra $\mathfrak{so}_7\oplus\mathfrak{sl}_2$. Based on Remark~\ref{RemBC} we can thus expect to apply some of our results to this case. Note that there are also many examples of algebra $\widetilde{\fg}$ as above which are far from simple, see for instance the {\em generalized Takiff superalgebras} introduced in \cite[Section 3.3]{GM}.

For Lie superalgebras of type $Q$, see \cite{Musson, ChWa}, the statement in Lemma~\ref{LemSimpO} has to be adapted, but an explicit classification of simple modules is known, see e.g., 
\cite[Section 1.5.4]{ChWa}. The category $\widetilde\cO$ for type $Q$ Lie superalgebras has been intensively studied in e.g. \cite{Br,Fr,FM,BD}.

\section{Kac functor preserves finite type socle and finite type radical for Lie superalgebras of type I} \label{sec8}

In this section we prove the following observation which naturally connects to various previous parts of the paper. To state the most general result we work in the category $\widetilde{\fg}$-sMod of all supermodules.

\begin{theorem}\label{thm91}
Let $\widetilde{\fg}=\widetilde{\fg}_{-1}\oplus\widetilde{\fg}_0\oplus \widetilde{\fg}_1$ be a 
Lie superalgebra of type I with $\fg:=\widetilde{\fg}_0$ and $\fp:=\widetilde{\fg}_0\oplus \widetilde{\fg}_1$. Then the corresponding 
Kac functor 
\begin{displaymath}
\mathrm{K}:=\mathrm{Ind}^{\widetilde{\fg}}_{\fp}:
{\fg}\text{-}\mathrm{sMod}\to\widetilde{\fg}\text{-}\mathrm{sMod} 
\end{displaymath}
(where $\widetilde{\fg}_1$ acts trivially on $\fg$-modules)
sends modules with finite type socle (resp radical) to modules with 
finite type socle (resp. radical)
preserving the length of the socle (resp. radical).
\end{theorem}

\begin{proof}
We start with the claim about the socle. Let $V\in \fg\text{-}\mathrm{sMod}$ be a module with finite type socle.
Set $d:=\dim \widetilde{\fg}_{-1}$. 
Any $\widetilde{\fg}$-submodule of $\mathrm{K}(V)$ intersects
$\Lambda^{d}\widetilde{\fg}_{-1}\otimes V$, see \cite[Lemma~3.1]{CM}.
The space $\Lambda^d\widetilde{\fg}_{-1}\otimes S$, with $S$ a simple submodule of 
$V$, therefore generates a simple $\widetilde{\fg}$-module of $K(V)$ and moreover intersects
$\Lambda^d\widetilde{\fg}_{-1}\otimes V$ precisely in $\Lambda^d\widetilde{\fg}_{-1}\otimes S$.
This follows from considering $K(V)$ as a $\mZ$-graded module where elements of $\fg_{\pm 1}$
act as operators in degree $\pm1$.
The submodule generated by 
$\Lambda^d\widetilde{\fg}_{-1}\otimes \soc(V)$ is thus semisimple, essential and has same length as $\soc(V)$.


We proceed with the claim about radical. By Remark~\ref{RemRad}\eqref{RemRad.1} and \eqref{RemRad.3}
it follows that $K$ sends modules with finite type radical to modules with superfluous radical. For an arbitrary $\fg$-module $M$ and $\widetilde{\fg}$-module $N$, we have
$$\Hom_{\widetilde{\fg}}(K(M),N)\;\cong\;\Hom_{\fg}(M,N^{{\widetilde{\fg}}_1})\quad\mbox{with}\quad
N^{{\widetilde{\fg}}_1}=\{v\in N\,:\, \widetilde{\fg}_1v=0\}. 
$$
The above equation and \cite[Theorem~4.1(i)]{CM} imply that $S_L:=L^{{\widetilde{\fg}}_1}$ is a simple $\fg$-module for each
simple $\widetilde{\fg}$-module $L$, and the map $L\mapsto S_L$ induces a bijection of isomorphism classes
of simple modules. For an arbitrary $\fg$-module $M$ and a simple $\widetilde{\fg}$-module $L$ we thus have
$$\Hom_{\widetilde{\fg}}(K(M),L)\;\cong\;\Hom_{\fg}(M,S_L).$$
Consequently, if $M$ has finite type radical, the length of the top of $M$ and $K(M)$ coincide.
\end{proof}

As an immediate corollary, we obtain the following. We call a Borel subalgebra $\widetilde{\fb}$ of $\widetilde{\fg}$
{\em distinguished} if it is of the form $\fb\oplus\fg_{1}$, for a Borel subalgebra $\fb$ of $\fg$.
\begin{corollary}\label{cor92}
All Verma supermodules with respect to distinguished Borel subalgebras over all Lie superalgebras of type I have simple socle. 
\end{corollary}

\begin{proof}
Verma supermodules over Lie superalgebras of type I can be obtained from
Verma modules over Lie algebras using Kac functor. Therefore the claim follows
from Theorem~\ref{thm91} and \cite[Proposition~7.6.3(i)]{Di}.
\end{proof}

\vspace{2mm}

\noindent
CC:~Department of Mathematics, Uppsala University, Box 480,
SE-75106, Uppsala, SWEDEN;
E-mail: {\tt chih-whi.chen\symbol{64}math.uu.se}
\hspace{2cm} 

\noindent
KC:~School of Mathematics and Statistics, University of Sydney, NSW 2006,\\ AUSTRALIA;
E-mail: {\tt kevin.coulembier@sydney.edu.au} 
\vspace{2mm}

\noindent
VM:~Department of Mathematics, Uppsala University, Box 480, SE-75106, Uppsala, SWEDEN;
E-mail: {\tt  mazor@math.uu.se}

\end{document}